\documentclass[12pt,a4paper,reqno]{amsart}

\usepackage[utf8]{inputenc}
\usepackage{amsthm, textcomp, amsmath, amsfonts,  amssymb }
\usepackage[pdftex]{graphicx}
\usepackage{url}
\usepackage{subfig}
\usepackage[colorlinks]{hyperref}
\usepackage{comment}
\usepackage[a4paper,left=3cm,right=3cm]{geometry}
\usepackage{tcolorbox}
\usepackage{hyperref}
\allowdisplaybreaks[4]

\begin{document}

\title[Summability and speed of convergence in an ergodic theorem]{Summability and speed of convergence\\ in an ergodic theorem}

\author[L. Colzani]{Leonardo Colzani}
\address[L. Colzani]{Dipartimento di Matematica,
Universit\`a degli Studi di Milano \mbox{Bicocca}, via Roberto Cozzi 55, Milano, Italy)}
\email{leonardo.colzani@unimib.it}

\author[B. Gariboldi]{Bianca Gariboldi}
\address[B. Gariboldi]{Dipartimento di Ingegneria Gestionale, dell'Informazione e della Produzione,
Universit\`a degli Studi di Bergamo, Viale Marconi 5, Dalmine BG, Italy}
\email{biancamaria.gariboldi@unibg.it}

\author[A. Monguzzi]{Alessandro Monguzzi}
\address[A. Monguzzi]{Dipartimento di Ingegneria Gestionale, dell'Informazione e della Produzione,
Universit\`a degli Studi di Bergamo, Viale Marconi 5, Dalmine BG, Italy}
\email{alessandro.monguzzi@unibg.it}

\begin{abstract}
Given an irrational vector $\alpha$ in $\mathbb{R}^{d}$, a continuous function $f(x)$ on the torus $\mathbb{T}^{d}$ and suitable weights $\Phi(N,n)$ such that  $\sum_{n=-\infty}^{+\infty}\Phi(N,n)=1$, we estimate the speed of convergence to the integral $\int_{\mathbb{T}^{d}}f(y)dy$ of the weighted sum $\sum_{n=-\infty}^{+\infty}\Phi(N,n)  f(x+n\alpha)$ as $N\rightarrow +\infty$.
Whereas for the arithmetic means 
$N^{-1}\sum_{n=1}^{N}f(x+n\alpha)$
the speed of convergence is never faster than $cN^{-1}$, for other means such speed can be accelerated. 
We estimate the speed of convergence in two theorems with different flavor. The first result is a metric one, and it provides an estimate of the speed of convergence in terms of the Fourier transform of the weights $\Phi(N,n)$ and the smoothness of the function $f(x)$ which holds for almost every $\alpha$. The second result is a deterministic one, and the speed of convergence is estimated also in terms of the Diophantine properties of the given irrational vector $\alpha\in\mathbb R^d$.

\end{abstract}

\subjclass[2010]{37A44, 37A46, 42B08}

\keywords{Kronecker sequences, Weyl sums, Ergodic theorem}
\thanks{The authors are members of GNAMPA - Istituto Nazionale di Alta Matematica (INdAM). The third author is partially supported by the Hellenic Foundation for Research and Innovation (H.F.R.I.) under the
“2nd Call for H.F.R.I. Research Projects to support Faculty Members \& Researchers” (Project Number: 73342).}
\maketitle

\newtheorem{Theorem}{Theorem}[section]
\newtheorem{Cor}[Theorem]{Corollary}
\newtheorem{lemma}[Theorem]{Lemma}
\newtheorem{definition}[Theorem]{Definition}
\newtheorem{prop}[Theorem]{Proposition}
\newtheorem{ex}[Theorem]{Example}
\theoremstyle{remark}
\newtheorem{remark}[Theorem]{Remark}

\newcommand{\nb}[3]{{\colorbox{#2}{\bfseries\sffamily\scriptsize\textcolor{white}{#1}}}
{\textcolor{#2}{\sf\small\textit{#3}}}}
\newcommand{\bianca}[1]{\nb{Bianca}{blue}{#1}}
\newcommand{\ale}[1]{\nb{Ale}{red}{#1}}
\newcommand{\leo}[1]{\nb{Leo}{cyan}{#1}}
\newcommand{\vecchio}[1]{\nb{vecchio}{brown}{#1}}

\section{Introduction}

 The motivation for this work comes from an attempt to estimate the speed of convergence in a classical ergodic theorem which we now describe. There are several results in the literature concerning this problem and here we only cite a few.
 
A classical result of L. Kronecker states that if $\alpha=(\alpha_1,\ldots,\alpha_d)\in\mathbb R^d$ is an irrational vector, that is, if $1,\alpha_1,\ldots,\alpha_d$ are linearly independent over the rationals, then the sequence $\{n\alpha\}_{n=1}^{+\infty}$ is dense in the torus $\mathbb{T}^{d}=\mathbb{R}^{d}/\mathbb{Z}^{d}$. This implies that for every continuous nonconstant function $f(x)$ on the torus the sequence $\{f(x+n\alpha)\}_{n=1}^{+\infty}$ does not have a limit as $n\to +\infty$. Another classical result obtained independently by P. Bohl, W. Sierpinski and H. Weyl states that the sequence $\{n\alpha\}_{n=1}^{+\infty}$ is uniformly distributed in the torus, and the arithmetic means of the sequence $\{f(x+n\alpha)\}_{n=1}^{+\infty}$ converge to the integral of the function,
\[
\lim_{N\to +\infty}\left\{  \dfrac{1}{N}\sum_{n=1}^{N}f(x+n\alpha)  -\int_{\mathbb{T}^{d}}f(y)dy\right\}  =0.
\]

For such classical facts we refer the reader, for instance, to \cite{KN} and \cite[Chapter 6]{Travaglini}. 
The map $x\mapsto x+\alpha$ is a measure preserving ergodic transformation whenever $\alpha\in\mathbb R^d$ is an irrational vector and the above results are particular cases of classical ergodic theorems. It is known that no general statement can be made about the rate of convergence in these theorems. In \cite{KP} and \cite{Krengel} it is proved that if $T$ is a measure preserving ergodic transformation of the interval $[0,1]$ and if $\{\varepsilon_{n}\}_{n=1}^{+\infty}$ is a positive sequence converging to 0, then there exists a continuous function $f(x) $ such that, for almost every $x$, one has 
\[
\limsup_{N\to +\infty}\left\{  \varepsilon_{N}^{-1}\left(  \dfrac{1}{N}\sum_{n=1}^{N}f(T^n(x))-\int_{0}^{1}f(y)dy \right) \right\}  =+\infty.
\]

Confirming a conjecture of Erd\"{o}s and Sz\"{u}s, in \cite{Kesten} and in \cite{Petersen} it is proved that if $f(x)$ is the characteristic function of an interval $\{  a\leq x\leq b\}  $, with $0<b-a<1$, then the quantity 
\[\sum_{n=1}^{N}f(x+n\alpha)  -N\int_{0}^{1}f(y)dy
\]
is bounded in $N$ if and only if $b-a=h\alpha-k$ for some integers $h$ and $k$. Therefore, for a characteristic function the speed of convergence $cN^{-1}$ is the exception, not the rule. For multidimensional analogues of such results see \cite{GL2, GL1}.

In \cite{HL} it is proved that if $f(x)  $ is a continuously differentiable function on $\{  0\leq x\leq 1\}  $ with $df(x)/dx$ Lipschitz continuous and with $f(  0) \neq f(  1)  $, then, for every $\alpha$, one has
\[
\limsup_{N\to+\infty}\left\{  \left| \sum_{n=1}^{N}f(n\alpha)  -N\int_{0}^{1}f(y)dy \right|\right\}  =+\infty.
\]

It is also proved that if $f(0)=f(1)$, hence $f(  x)  $ is continuous as a function on the torus $\mathbb{T}$, but the derivative may have a jump discontinuity, then, for almost every $\alpha$, one has
\[
\limsup_{N\to+\infty}\left\{  \sup_{0\leq x\leq 1}\left\{  \left| \sum_{n=1}^{N}f(x+n\alpha)  -N\int_{0}^{1}f(y)dy \right| \right\}  \right\}<+\infty.
\]

Observe that a discontinuous function cannot have an absolutely convergent Fourier expansion. On the other hand, the assumptions $f(0)  =f(  1)  $ and $df(x)/dx$ Lipschitz continuous in $\{  0\leq x\leq 1\}  $ imply that $\vert\widehat{f}(m)\vert \leq cm^{-2}$. More generally, if $df(x)/dx$ is H\"{o}lder continuous with exponent $\varepsilon>0$, then $\vert \widehat{f}(m)\vert \leq cm^{-1-\varepsilon}$. 
In \cite{DP} it is proved that if $\{  \alpha_{n}\}_{n=0}^{+\infty}$ is a van der Corput sequence on the interval $\{0\leq x\leq 1\}  $ and if the Fourier coefficients of the function $f(x)  $ have decay $\vert \widehat{f}(  m)\vert \leq c\vert m\vert ^{-1-\varepsilon}$ for some $\varepsilon>0$, then 
\[
\sup_{N\geq 1}\left\{\left| \sum_{n=1}^{N}f(  \alpha_{n})  -N\int_{0}^{1}f(x)dx\right|\right\} <+\infty.
\]

In \cite{BBH} it is proved that the expected speed of convergence of Weyl sums of continuous, or, more generally, square integrable functions, is slightly less than $N^{-1/2}$. More precisely, they proved that if $f(x)  $ is square integrable and if $\nu<1/2$ then, for almost every $(\alpha,x)  \in\mathbb{T}^{d}\times\mathbb{T}^{d}$, one has
\[
\limsup_{N\to+\infty}\left\{  N^{\nu}\left| \dfrac{1}{N}\sum_{n=1}^{N}f(x+n\alpha)  -\int_{\mathbb{T}^{d}}f(y)dy\right| \right\}=0.
\]

They also proved that the exponent $-1/2$ is best possible and that there exist continuous functions $f(x)$ such that, for almost every $(  \alpha,x)\in\mathbb{T}^{d}\times\mathbb{T}^{d}$, one has
\[
\limsup_{N\to+\infty}\left\{  N^{\frac{1}{2}}\left| \dfrac{1}{N}\sum_{n=1}^{N}f(x+n\alpha)  -\int_{\mathbb{T}^{d}}f(y)dy\right| \right\}  =+\infty.
\]

In conclusion, the rate of convergence of the means $N^{-1}\sum _{n=1}^{N}f(x+n\alpha)$ to the integral $\int _{\mathbb T^d}f(y)dy$ can be arbitrarily slow and it is also quite easy to see that this rate of convergence  cannot be faster than $cN^{-1}$; see the proof of Corollary \ref{C-1}. The goal in this paper is to show that, with suitable smoothness assumptions on the function $f(x)$, the speed of summability of the divergent sequence $\{f(x+n\alpha)\}_{n=1}^{+\infty}$ can be improved if instead of the arithmetic means one considers smoother means such as, for instance,
\[
\dfrac{1}{N}\sum_{n=1-N}^{N-1}\left(1-\dfrac{|n|}{N}\right)  f(x+n\alpha).
\]

See \cite{DT} and \cite{Z} for references about summation methods. Let us now fix some notations for what follows. Denote by 
$\Vert t\Vert $ the distance of a real number $t$ to the nearest integer, that is, $\|t\|=\inf_{n\in\mathbb Z}\left\{|t-n|\right\}$.
Functions on the torus $\mathbb{T}^{d}=\mathbb{R}^{d}/\mathbb{Z}^{d}$ are identified with periodic functions on $\mathbb{R}^{d}$ with period $\mathbb{Z}^{d}$. The Fourier transform and the Fourier expansion of an integrable function on the torus are defined respectively by
\begin{align*}
&\widehat{f}(m)  =\int_{\mathbb{T}^{d}}f(x)e^{-2\pi im\cdot x}  dx,\qquad   \qquad Sf(x)=\sum_{m\in\mathbb T^d}\widehat f(m) e^{2\pi i m\cdot x}.
\end{align*}

The Sobolev space $W^{\delta,2}(\mathbb{T}^{d})$, $\delta>0$, is the space of distributions on $\mathbb T^d$ defined by the norm
\[
\|f\|_{\delta, 2}= \left(\sum_{m\in\mathbb{Z}^{d}}
\left(1+\vert m\vert ^{2}\right)^{\delta}
| \widehat{f}(m) |^{2} \right)^{\frac{1}{2}}.
\]

In what follows $\Phi(N,n)$ denotes a complex valued function of the positive integer variable $N\geq 1$ and the integer variable $n\in\mathbb Z$, with the property that for every $N$ the function $n \rightarrow \Phi (N,n)$ has bounded support, and that
\begin{equation}
\label{condizione}
\sum _{n=-\infty}^{+\infty}\Phi(N,n)=1.
\end{equation}

The weighted discrepancy associated to the weights $\left\{\Phi(N,n)\right\}_{n =-\infty}^{+\infty}$ and to the Kronecker sequence $\{n\alpha\}_{n=-\infty}^{+\infty}$, with $\alpha\in\mathbb R^d$, or equivalently with $\alpha\in\mathbb{T}^d$, is defined by  
\[
\mathcal D^{\Phi,\alpha}_N f(x)=\sum_{n=-\infty}^{+\infty}\Phi(N,n)  f(x+n\alpha)  -\int_{\mathbb{T}^{d}}f(y)dy.
\]

An example to keep in mind is
\[\Phi(N,n) = \left(\sum _{n=-\infty}^{+\infty}\Psi(N^{-1}n)  \right)^{-1}\Psi(N^{-1}n)\]
where $\Psi(t)$ is a suitable bounded function with compact support. In this case $N$ is roughly the size
of the support of the function $n\to \Phi(N,n)$. The assumption of compact support
could be weakened assuming a suitably fast decay at infinity.

Our first main result is related to the results in \cite{Colzani, Colzani2} and it reads as follows.

\begin{Theorem}\label{T-1}
Let $\mathcal D^{\Phi, \alpha}_N$ be the operator defined as above, with $\Phi(N,n)$ satisfying \eqref{condizione}. Assume the following.
\begin{itemize}
\item[$(i)$] There exist constants $K>0$ and $\vartheta >0$ such that for every $N\geq 1$ and every $t\in\mathbb R$ one has
\[
\left\vert \sum _{n=-\infty}^{+\infty}\Phi(N,n)e^{2\pi int}\right\vert \leq K(1+N\Vert t\Vert)^{-\vartheta}.
\]
\item[$(ii)$] The function $f(x)$ is in the Sobolev class $W^{\delta,2}(\mathbb T^d)$, with $\delta>d/2$ if $0<\vartheta<1$, and $\delta>d\vartheta-d/2$ if $\vartheta\geq 1$.
\end{itemize}

Then, for almost every $\alpha$ there exists a positive constant $c(f,\alpha)$ such that for every positive integer $N$ one has
\[
\sup_{x\in\mathbb{T}^{d}}\left\{\left|\mathcal D^{\Phi, \alpha}_N f(x) \right|\right\}\leq c(f,\alpha)N^{-\vartheta}.
\]
\end{Theorem}

Observe that if the assumption $(i)$ holds true with an exponent $\vartheta_0$, then it also holds true for every $\vartheta_1 \leq \vartheta_0$. 
If the function $f(x)$ has a degree of smoothness $d/2< \delta \leq d\vartheta_0 -d/2$ with $\vartheta_0> 1$, then one cannot guarantee a speed of convergence $cN^{-\vartheta_0}$, but at least one can guarantee a speed $cN^{-\vartheta_1}$ for every $1<\vartheta_1< \vartheta_0$ with $\delta> d\vartheta_1 -d/2$. 

The above result is a metric one and it holds true for almost every $\alpha$. Our second main result is a deterministic one and it holds for a specific $\alpha$.
\begin{Theorem}\label{T-2}
Let $\mathcal D^{\Phi, \alpha}_N$ be the operator defined as above with $\Phi(N,n)$ satisfying \eqref{condizione}. Assume the following.
\begin{itemize}
\item[$(i)$] There exist constants $K>0$ and $\vartheta>0$ such that for every positive integer $N$ and every $t\in\mathbb R$ one has 
\[
\left\vert \sum _{n=-\infty}^{+\infty}\Phi(N,n)e^{2\pi int}\right\vert \leq K(1+N\Vert t\Vert)^{-\vartheta}.
\]
\item[$(ii)$] The vector $\alpha\in\mathbb R^d$ is irrational and there exist constants $H>0$ and $\sigma\geq d$ such that $\Vert \alpha\cdot m\Vert \geq H\vert m\vert^{-\sigma}$ for every $m\in\mathbb{Z}^{d}\setminus\{0\}$.
\item[$(iii)$] Finally assume that $\delta>d/2$ and set
\begin{align*}
&X(d,\delta,\vartheta,\sigma,N)  =\begin{cases}
N^{\frac{-\delta+\vartheta\sigma-d\vartheta+d/2}{\sigma-d}} & \text{if }\vartheta<1/2 \text{ and } \delta<\vartheta\sigma-d(\vartheta-1/2)  ,\\
\log^{\frac 12}(1+N)  & \text{if } \vartheta<1/2 \text{ and } \delta=\vartheta \sigma-d(\vartheta-1/2),\\
1 & \text{if } \vartheta<1/2\text{ and } \delta>\vartheta\sigma-d(\vartheta-1/2),\\
N^{\frac{\sigma/2-\delta}{\sigma-d}}\log^{\frac 12}(1+N)  &\text{if } \vartheta=1/2 \text{ and } \delta<\sigma/2,\\
\log(1+N)  &\text{if } \vartheta=1/2 \text{ and } \delta=\sigma/2,\\
1 & \text{if } \vartheta=1/2 \text{ and }  \delta>\sigma/2,\\
N^{\frac{\vartheta(\vartheta\sigma-\delta)}{\vartheta\sigma-d/2}} &\text{if } \vartheta>1/2 \text{ and } \delta<\vartheta\sigma,\\
\log^{\frac 12}(1+N)  & \text{if } \vartheta>1/2 \text{ and } \delta=\vartheta\sigma,\\
1 & \text{if } \vartheta>1/2 \text{ and } \delta>\vartheta\sigma.
\end{cases}
\end{align*}
\end{itemize}

Then there exists a positive constant $c=c(H,K,d,\delta,\vartheta,\sigma)$ such that for every function $f(x)$ in the Sobolev space $W^{\delta,2}(\mathbb{T}^{d})$ and every positive integer $N$ one has
\[
\sup_{x\in\mathbb{T}^{d}}\left\{\left|\mathcal D^{\Phi, \alpha}_N f(x)\right|\right\}\leq cN^{-\vartheta}X(d,\delta,\vartheta,\sigma,N)\|f\|_{\delta,2}.
\]
\end{Theorem}

Observe that both Theorem \ref{T-1} and Theorem \ref{T-2} guarantee a speed of convergence $cN^{-\vartheta}$, up to some possible logarithmic transgressions, but the smoothness assumptions on the functions in these theorems are different. The index of smoothness $\delta>d\vartheta -d/2$ in Theorem \ref{T-1} is allowed to be smaller than the index $\delta>\vartheta\sigma\geq d\vartheta$ in Theorem \ref{T-2}. On the other hand the conclusion in Theorem \ref{T-1} holds for almost every $\alpha$, with $\alpha$ depending on the given function one is considering, whereas in Theorem \ref{T-2} the vector $\alpha$ is independent of the function.
Anyhow, both theorems are essentially sharp.
The following theorem shows that in Theorem \ref{T-1} and in Theorem \ref{T-2} the speed of convergence $cN^{-\vartheta}$ cannot be accelerated for every nonconstant function, provided that the assumption $\left\vert \sum _{n=-\infty}^{+\infty}\Phi(N,n)e^{2\pi int}\right\vert \leq K(1+N\Vert t\Vert)^{-\vartheta}$ can be reversed. 

\begin{Theorem}\label{T-3}
Set
\[
C(t)=\limsup_{N\to+\infty} \left\{N^{\vartheta}\left| \sum _{n=-\infty}^{+\infty}\Phi(N,n)e^{2\pi int}
\right|\right\}.
\]

Then, for every function $f(x)$, every $m\in\mathbb{Z}^d\setminus \{ 0\}$ and every $\alpha \in \mathbb{T}^d$ one has  
\begin{align*}
\limsup_{N\to+\infty}\left\{N^{\vartheta}\sup_{x\in \mathbb{T}^d}\left\{\left|\mathcal D^{\Phi, \alpha}_N f(x) \right|\right\}\right\}
&\geq C(m\cdot\alpha)|\widehat{f}(m)|.
\end{align*}
\end{Theorem}

In particular, if $C(t)>0$ for a set of $t\in \mathbb{T}$ of measure $0<\eta\leq1$, then $C(m\cdot \alpha)>0$ for a set of $\alpha\in \mathbb{T}^d$ of measure $\eta$; see Lemma \ref{L-1}. 

Notice that in this theorem the smoothness index of the function plays no role. Nonetheless, some smoothness is necessary. 
Indeed, since the Sobolev space $W^{\frac{d}{2},2}(\mathbb{T}^d)$ contains unbounded functions, it easily follows that the smoothness assumption $\delta >d/2$ in Theorem \ref{T-1} and Theorem \ref{T-2} is necessary.
\begin{Theorem}\label{T-4}                 
There exists a function $f(x)$ in the Sobolev space $W^{\frac{d}{2},2}(\mathbb T^d)$ such that for every irrational vector $\alpha$ and every $N$ one has
\[
\sup_{x\in\mathbb{T}^{d}}\left\{\left|\mathcal D^{\Phi, \alpha}_N f(x) \right|\right\} =+ \infty.
\]

Moreover, if the sequence $\left\{\Phi(N,n)\right\}_{n=-\infty}^{+\infty}$ is non-negative, then the above discrepancy is infinite for every $\alpha$.
\end{Theorem}

The following theorem shows that the index $\delta > d\vartheta -d/2$ in Theorem \ref{T-1} is sharp, provided that the assumption $(i)$ in the theorem can be reversed.

\begin{Theorem}\label{T-5}
Assume that for an infinite sequence of $N$'s there exists $H>0$ such that
\[
 \sum _{n=-\infty}^{+\infty}\Phi(N,n)e^{2\pi int} \geq H(1+N\Vert t\Vert)^{-\vartheta}.
\]
\begin{itemize} 
\item[$(i)$] If $\delta<d\vartheta-d/2$ then there exists a function $f(x)$ in the Sobolev space $W^{\delta,2}(\mathbb T^d)$ such that, for every $\alpha$, one has
\[
\limsup_{N\to+\infty} \left\{N^{\vartheta}
\sup_{x\in\mathbb{T}^{d}}\left\{\left|\mathcal D^{\Phi, \alpha}_N f(x) \right|\right\} \right\}=+ \infty.
\]
\item[$(ii)$]  There exists a function $f(x)$ in the Sobolev space $W^{d\vartheta -\frac{d}{2},2}(\mathbb T^d)$ such that, for almost every $\alpha$, one has
\[
\limsup_{N\to+\infty} \left\{N^{\vartheta}
\sup_{x\in\mathbb{T}^{d}}\left\{\left|\mathcal D^{\Phi, \alpha}_N f(x) \right|\right\} \right\}=+ \infty.
\]
\end{itemize}

\end{Theorem}

The following theorem shows that the smoothness index $\delta > \vartheta \sigma$ in Theorem \ref{T-2} is sharp, provided that the assumptions $(i)$ and $(ii)$ in the theorem can be
reversed.
 
\begin{Theorem}\label{T-6} 
Assume that for an infinite sequence of $N$'s there exists $H>0$ such that
\[
 \sum _{n=-\infty}^{+\infty}\Phi(N,n)e^{2\pi int} \geq H(1+N\Vert t\Vert)^{-\vartheta}.
\]

Assume also that for a given $\alpha$ there exist $L>0$ and an infinite subset $\Omega$ of $ \mathbb Z^d\backslash\{0\}$ with 
$\Vert \alpha\cdot m\Vert \leq L\vert m\vert^{-\sigma}$ 
for every $m \in \Omega$.
Then there exists a function $f(x)\in W^{\vartheta \sigma,2}(\mathbb{T}^d)$ such that
\[
\limsup_{N\to+\infty} \left\{N^{\vartheta}\sup_{x\in \mathbb{T}^d} \left\{ \left| \mathcal D^{\Phi,\alpha}_N f(x) \right|  \right\}\right\} =+\infty.
 \]
\end{Theorem}

A crucial assumptions in Theorem \ref{T-2} is the behavior of the  sequence $\{\Vert \alpha\cdot m\Vert\}_{m\in\mathbb Z^d}  $. We recall some results in Diophantine approximation.
\begin{itemize}
\item[$(\star)$]  A classical result of Dirichlet states that for every vector $\alpha\in\mathbb{R}^d$ and every positive integer $M$ there exists $m=(m_{1},m_{2},\dots,m_{d})$ in $\mathbb{Z}^{d}$ with $\vert m_{j}\vert \leq M$ for every $j=1,\ldots,d$, and with $\Vert \alpha\cdot m\Vert \leq M^{-d}$. In particular, if there exists $H>0$ such that $\Vert \alpha\cdot m\Vert \geq H\vert m\vert ^{-\sigma}$ for every $m\in \mathbb{Z}^{d}\setminus\{  0\}$, then $\sigma\geq d $. See e.g. \cite[Chapter II, Theorem 1E]{Schmidt}.
\item[$(\star \star)$] If $\{1,\alpha_{1},\alpha_{2},\ldots,\alpha_{d}\}$ is a basis of a real algebraic number field of degree $d+1$, and if $\alpha=(\alpha_{1},\alpha_{2},\ldots,\alpha_{d})$, then there exists $H>0$ such that $\Vert \alpha\cdot m\Vert \geq H\vert m\vert^{-d}$ for every $m\in\mathbb{Z}^{d}\setminus\{  0\}$. See e.g. \cite[Chapter 2, Theorem 4A]{Schmidt}.
\item[$(\star \star \star)$] For every $\sigma>d$ the set of vectors $\alpha$ in $\mathbb{R}^{d}$ with $\Vert \alpha\cdot m\Vert <\vert m\vert^{-\sigma}$ for infinitely many $m\in\mathbb{Z}^{d}\setminus\{  0\}  $ has measure zero. This result is due to Khintchine. See e.g. \cite[Chapter 3, Theorem 3A]{Schmidt}, \cite{BeresnevichVelani} and the references therein. 
\end{itemize}

It follows from these results that the assumption $(ii)$ in Theorem \ref{T-2} is not empty  and $\sigma\geq d$ is necessary. The following corollaries show that also the assumption $(i)$ in Theorem \ref{T-1} and Theorem \ref{T-2} is nonempty. 

\begin{Cor}\label{C-1}
Let
 $ \Phi(N,n) = \left(\sum _{n=-\infty}^{+\infty}\Psi(N^{-1}n)  \right)^{-1}\Psi(N^{-1}n)$, with
\[
\Psi(t)=\begin{cases}
1 & \text{if } \vert t\vert \leq 1,\\
0 & \text{if } \vert t\vert >1.
\end{cases}
\]

Then Theorem \ref{T-1}, Theorem \ref{T-2} and Theorem \ref{T-3} apply with every $\vartheta \leq 1$.
\end{Cor}

\begin{Cor}\label{C-2}
Let $ \Phi(N,n) = \left(\sum _{n=-\infty}^{+\infty}\Psi(N^{-1}n)  \right)^{-1}\Psi(N^{-1}n)$, with
\[
\Psi(t)  =\begin{cases}
1-\vert t\vert  & \text{if } \vert t\vert \leq1,\\
0 & \text{if } \vert t\vert >1.
\end{cases}
\]

Then Theorem \ref{T-1}, Theorem \ref{T-2} and Theorem \ref{T-3} apply with every $\vartheta \leq 2$.
\end{Cor}

\begin{Cor}\label{C-3}
Let $\gamma>0$ and let $ \Phi(N,n) = \left(\sum _{n=-\infty}^{+\infty}\Psi(N^{-1}n)  \right)^{-1}\Psi(N^{-1}n)$, with 
\[
\Psi(t)  =\begin{cases}
(1-\vert t\vert^{2})^{\gamma} & \text{if } \vert t\vert \leq1,\\
0 & \text{if } \vert t\vert > 1.
\end{cases}
\]

Then Theorem \ref{T-1}, Theorem \ref{T-2} and Theorem \ref{T-3} apply with every $\vartheta \leq \gamma+1$.
\end{Cor}

\begin{Cor}\label{C-4}
Let 
\[
\Phi(N,n)  =\begin{cases}
\dfrac{(2N)!}{2^{2N} \left( N-n \right)! \left( N+n \right)! } & \text{if } 
\vert n\vert \leq N!,\\
0 & \text{if } \vert n\vert > N!.
\end{cases}
\]

Then Theorem \ref{T-1} and Theorem \ref{T-2} can be applied with every $\vartheta$, but with $[\sqrt{N}]$ instead of $N$, that is the relations between the indexes $d$, $\delta$, $\vartheta$ and $\sigma$ are the ones in the theorems, but the speed of convergence is $cN^{-\frac{\vartheta}{2}}$ instead of $cN^{-\vartheta}$.
\end{Cor}

The next corollary shows that, up to a small logarithmic transgression, Kronecker sequences associated to vectors $\alpha$ which satisfy the hypothesis $(ii)$ in Theorem \ref{T-2} with $\sigma=d$ give optimal quadrature rules for Sobolev functions. See \cite{BCCGST} for results about quadrature rules for Sobolev functions. See also \cite{EEGGP} and \cite{GG} for results about existence of optimal quadrature rules.

\begin{Cor}\label{C-5}
Let $ \Phi(N,n) = \left(\sum _{n=-\infty}^{+\infty}\Psi(N^{-1}n)  \right)^{-1}\Psi(N^{-1}n)$, with  $\Psi(t)$ a  smooth compactly supported bounded function satisfying 
$cN\leq \left| \sum _{n=-\infty}^{+\infty}\Psi(N^{-1}n) \right| \leq CN$. Let $\alpha$ be an irrational vector in $\mathbb{R}^{d}$ and assume also that there exist constants $H>0$ and $\sigma\geq d$ such that $\Vert \alpha\cdot m\Vert \geq H\vert m\vert ^{-\sigma}$ for every $m\in\mathbb{Z}^{d}\setminus\{  0\}  $. Finally assume that $\delta>d/2$ and set
\begin{align*}
&Y(d,\delta, \sigma,N)  =\begin{cases}
N^{\frac{d}{\sigma-d}(\frac12-\frac\delta\sigma)}& \text{if }\delta/\sigma<1/2,\\
\log(1+N) &\text{if }\delta/\sigma=1/2,\\
\log^{\frac{1}{2}}(1+N)  &\text{if } \delta/\sigma>1/2.
\end{cases}
\end{align*}

Then there exists a positive constant $c$ such that
\[
\sup_{x\in\mathbb T^d}\left\{ \left|\mathcal D^{\Phi, \alpha}_Nf(x) \right|\right\}\leq c\|f\|_{\delta, 2} N^{-\frac{\delta}{\sigma}} Y(d,\delta,\sigma, N).
\]

When $\sigma=d$ the speed of convergence $cN^{-\frac{\delta}{d}}$ cannot be improved in the sense that there exists $c>0$ such that for every distribution of points $\{z(n)\}_{n=1}^{N}$ and weights $\{\omega(n)\}_{n=1}^{N}$ there exist nonconstant functions in $W^{\delta,2}(\mathbb{T}^{d})$ with 
\[
\left|\sum _{n=1}^{N}\omega(n)f(z(n))  -\int _{\mathbb{T}^{d}}f(y)dy\right| \geq cN^{-\frac{\delta}{d}}\Vert f\Vert _{\delta, 2}.
\]
\end{Cor}

The above corollaries show that it is quite easy to exhibit examples of weights $\Phi(N,n)$ that satisfy the assumptions in Theorem \ref{T-1} and Theorem \ref{T-2}. It is less immediate to construct weights $\Phi(N,n)$ that satisfy the reverse assumption, in particular the ones in Theorem \ref{T-5} and Theorem \ref{T-6}. However, such weights exist; see Remark \ref{R-1}. We include in the paper an appendix where we shall consider the logarithmic means defined by the weights 
 \[ \Phi(N,n) =\begin{cases} \left(\displaystyle \sum\limits_{m=1}^{N}\dfrac{1}{m} \right)^{-1} \dfrac{1}{n} & \text{if } 1\leq n\leq N\\  0 & \text{otherwise}.\end{cases}\] 
Although these logarithmic means do not satisfy exactly the assumptions in Theorem \ref{T-1} and Theorem \ref{T-2}, the proofs of these theorems can be adapted.

To conclude, the above results may have continuous analogues where the discrete means are replaced by continuous means,
\[
\mathcal C^{\Phi, \alpha}_T f(x)= \int _{-\infty}^{+\infty}\Phi(T,t)  f(x+t\alpha)dt-\int _{\mathbb{T}^{d}}
f(y)dy.
\]

We plan to investigate such operator in future works.

\medskip

In the next section we provide the proofs of our main theorems, whereas in Section \ref{final_section} we conclude with some final remarks.

\section{Proofs of the main results}
To prove Theorem \ref{T-1} we need an elementary lemma.

\begin{lemma}\label{L-1}
 If $g(t)$ is a periodic locally integrable function on $\mathbb T$, then, for every $m\in\mathbb Z^d\backslash\{0\}$, one has
 \[
\int_{\mathbb T^d}g(m\cdot\alpha)\, d\alpha=\int_{\mathbb T}g(t)\, dt.
 \]

More precisely, if $g(t)$ is a measurable function on $\mathbb T$, then, for every $m\in\mathbb Z^d\backslash\{0\}$, the functions $g(t)$, $t\in\mathbb{T}$, and $g(m\cdot\alpha)$, $\alpha\in\mathbb{T}^d$, have the same distribution function. Namely, for every $-\infty<s<+\infty$,
\[
\left\vert \left\{t\in\mathbb{T}, \ g(t)>s \right\} \right\vert
= \left\vert \left\{\alpha\in\mathbb{T}^d, \ g(m\cdot\alpha)>s \right\} \right\vert.
\]
\end{lemma}

\begin{proof}
 By periodicity and a change of variables, for every non-zero integer $h$ and real number $u$, one has 
 \[
  \int_{\mathbb T}g(ht+u)\, dt=\int_{\mathbb T} g(t)\, dt.
 \]

 Hence, setting $m=(h,k)$ with $h\in\mathbb Z\backslash\{0\}$ and $k\in\mathbb Z^{d-1}$, and $\alpha=(\beta,\gamma)\in\mathbb R\times\mathbb R^{d-1}$, one obtains
\begin{align*}
\int_{\mathbb T^d}g(m\cdot\alpha)\, d\alpha&=\int_{\mathbb T^{d-1}}\bigg(\int_{\mathbb T}g(h\beta+k\cdot\gamma)\, d\beta\bigg)\, d\gamma\\
&=\int_{\mathbb T^{d-1}}\bigg(\int_{\mathbb T} g(t)\, dt\bigg)\, d\gamma=\int_{\mathbb T}g(t)\, dt.
\end{align*}

The distribution functions of $g(t)$ and $g(m\cdot\alpha)$ are seen to be equal by applying the above identity to the characteristic function of the upper level set
$\chi_{\{t\in\mathbb{T},\ g(t)>s\}}(t)$. 
\end{proof}
We now prove our first main result.
\begin{proof}[Proof of Theorem \ref{T-1}]
First observe that for every $0<p< 2$ one has
\[
\sum_{m\in\mathbb Z^d}|\widehat f(m)|^{p}
\leq \left(\sum_{m\in\mathbb Z^d} | \widehat f(m) |^2 \left(1+|m|^2 \right)^{\delta}\right)^{\frac{p}{2}}\left(\sum_{m\in\mathbb Z^d} \left(1+|m|^2 \right)^{-p \delta/(2-p)}\right)^{\frac{2-p}{2}}.
\]

The first factor is the Sobolev norm of $f(x)$, whereas the second series converges provided that $2p\delta/(2-p)>d$. In particular, for $p=1$ and $\delta>d/2$ one sees that the Fourier expansion of $f(x)$ converges absolutely. This fact and the compact support of $n \rightarrow \Phi(N,n)$ assure the pointwise identity
\begin{align*}
\mathcal D^{\Phi, \alpha}_Nf(x)
&=\sum _{n=-\infty}^{+\infty}\Phi(N,n)f(x+n\alpha)-\int _{\mathbb{T}^{d}}f(y)dy\\
&=\sum _{m\in\mathbb{Z}^{d}\setminus\{ 0\}} 
\left(  \sum _{n=-\infty}^{+\infty} \Phi(N,n)e^{2\pi inm\cdot\alpha} \right) \widehat{f}(m)e^{2\pi im\cdot x}.
\end{align*}

Hence, thanks to $(i)$, one has
\[
\left| \mathcal D^{\Phi,\alpha}_Nf(x) \right| \leq  KN^{-\vartheta}\sum_{m\in\mathbb Z^d\backslash\{0\}}| \widehat f(m) | \|m\cdot\alpha\|^{-\vartheta} =c(f,\alpha)N^{-\vartheta}.
\]

In order to show that the constants $c(f,\alpha)$ are finite for almost every $\alpha$ it suffices to show that the series defining these constants converges absolutely for almost every $\alpha$. By the previous lemma the functions $\alpha\mapsto \|m\cdot \alpha\|^{-\vartheta}$ are in $L^p(\mathbb T^d)$ for every $p<1/\vartheta$ with norm independent of $m$,
\[ \int_{\mathbb{T}^d} \left(\| m\cdot\alpha\|^{-\vartheta}\right)^{p}d\alpha=\int_{\mathbb{T}} \| t\|^{-\vartheta p}dt=\int_{0}^{\frac{1}{2}} t^{-\vartheta p}dt=\dfrac{2^{\vartheta p-1}}{1-\vartheta p}.\]

If $0<\vartheta<1$, then the functions the functions $\alpha\mapsto \|m\cdot \alpha\|^{-\vartheta}$ are integrable and the series
\[
\sum_{m\in\mathbb Z^d\setminus\{0\}}|\widehat f(m)|\|m\cdot\alpha\|^{-\vartheta}
\]
converges provided that 
\[
\sum_{m\in\mathbb Z^d}|\widehat f(m)|<\infty.
\]
As observed before, this holds true if $\delta>d/2$.

\medskip

If $\vartheta\geq 1$, then $0<p<1/\vartheta\leq1$, and, by the inequality $|a+b|^p\leq |a|^p+|b|^p$, the series
\[
\sum_{m\in\mathbb Z^d\setminus\{0\}}|\widehat f(m)|\|m\cdot\alpha\|^{-\vartheta}
\]
converges for almost every $\alpha$ and in the $L^p(\mathbb T^d)$ quasinorm provided that 
\[
\sum_{m\in\mathbb Z^d}|\widehat f(m)|^p<\infty.
\]

As observed at the beginning of the proof this happens  for every $0<p<2$ whenever  $2p\delta/(2-p)>d$, from which one obtains $\delta>d(1/p-1/2)>d\vartheta-d/2$.
\end{proof}

A key ingredient in the proof of Theorem \ref{T-2} is a classical result in Diophantine approximation. Let $\gamma$ be an irrational number. If the sequence $\{\Vert \gamma n\Vert\}_{n=1}^{N}$ is well-distributed in $0\leq t\leq1/2$ as it is distributed the sequence $\{  n/(2N)\}_{n=1}^{N}$, then one can guess that
\[
\sum_{n=1}^{N}\Vert \gamma n\Vert ^{-\vartheta}\approx\sum _{n=1}^{N}\left(\dfrac{n}{2N}\right)^{-\vartheta}\leq\begin{cases}
cN & \text{if }0<\vartheta<1,\\
cN\log(N)  & \text{if }\vartheta=1,\\
cN^{\vartheta} & \text{if }1<\vartheta<+\infty.
\end{cases}
\]

Under suitable Diophantine assumptions on $\gamma$ the above conjectured estimate is correct. The following lemma is a variant of known results (see e.g. \cite[Chapter 3]{Lang}).

\begin{lemma}\label{L-2}
Assume that $\alpha=\left(\alpha_1,\ldots,\alpha_d\right)\in\mathbb R^d$ is an irrational vector, that is, $1,\alpha_1,\ldots,\alpha_d$ are linearly independent over the rationals, and assume that there exist constants $H>0$ and $\sigma\geq d$ such that $\Vert \alpha\cdot m\Vert \geq H\vert m\vert^{-\sigma}$ for every $m\in\mathbb{Z}^{d}\setminus\{  0\}$. Then there exists a positive constant $c$ such that, for every $R\geq1$,
\[\sum _{0<\vert m\vert <R}\Vert \alpha\cdot m\Vert ^{-\vartheta}\leq\begin{cases}
cR^{\vartheta\sigma+d(1-\vartheta)} & \text{if } 0<\vartheta<1,\\
cR^{\sigma}\log(R)  & \text{if }\vartheta=1,\\
cR^{\vartheta\sigma} & \text{if }1<\vartheta<+\infty.
\end{cases}
\]
\end{lemma}

\begin{proof} By the assumptions $\Vert \alpha\cdot m\Vert \geq H\vert m\vert ^{-\sigma}$ and $\vert m\vert <R$, the interval $[0,H/(2R)^{\sigma})$ does not contain any term of the sequence $\{\Vert \alpha\cdot m\Vert \}_{0<\vert m\vert <R}$. Moreover, for every integer $n$ such that $0<n<2^{\sigma-1}H^{-1}R^{\sigma}$ the interval $I_{n ,R}^{\sigma}=[nH/(2R)^{\sigma},(n+1)H/(2R)^{\sigma})$ contains at most one term of such sequence. Indeed, if there are two terms in the interval, then there are integer points $p\neq q$ with $\vert p\vert, \vert q\vert <R$, and integers $u$ and $v$ such that
\[
\vert (\alpha\cdot p-u)\pm (\alpha\cdot q-v)\vert <\dfrac{H}{(2R)^{\sigma}}.
\]

The signum is minus if $\alpha\cdot p$ and $\alpha\cdot q$ approximate the nearest integers $u$ and $v$ both from above or both from below, the signum is plus if one approximation is from above and the other from below. Hence,
\[
\Vert \alpha\cdot(p\pm q)\Vert <\dfrac{H}{(2R)^{\sigma}}.
\]

But $\vert p\pm q\vert <2R$, and this contradicts the assumption $\Vert \alpha\cdot m\Vert \geq H\vert m\vert ^{-\sigma}$. 
Notice that the number of intervals $I_{n ,R}^{\sigma}$'s is of the order of $cR^{\sigma}$, whereas the number of integer points in the punctured ball $\{0<\vert m\vert <R\}  $ is about $cR^{d}$, and recall also that $\sigma\geq d$. Observe that one has the worst estimate when the terms of the sequence $\{\Vert\alpha\cdot m\Vert\}_{0<\vert m\vert <R}$ are concentrated in the the first $0<n\leq cR^{d}$ intervals. In conclusion,
\[\sum _{0<\vert m\vert <R} \Vert \alpha\cdot m\Vert^{-\vartheta}\leq\sum_{0<n\leq cR^{d}}\left(\dfrac{nH}{(2R)^{\sigma}}\right)^{-\vartheta}\leq
\begin{cases}
cR^{\vartheta\sigma+d(1-\vartheta)} & \text{if }0<\vartheta<1,\\
cR^{\sigma}\log(1+R)  & \text{if }\vartheta=1,\\
cR^{\vartheta\sigma} & \text{if }1<\vartheta<+\infty.
\end{cases}
\]
\end{proof}

\begin{proof}[Proof of Theorem \ref{T-2}]
As in the proof of Theorem \ref{T-1} one has the pointwise identity
\begin{align*}
\mathcal D^{\Phi, \alpha}_Nf(x)
&=\sum _{n=-\infty}^{+\infty}\Phi(N,n)f(x+n\alpha)-\int _{\mathbb{T}^{d}}f(y)dy\\
&=\sum _{m\in\mathbb{Z}^{d}\setminus\{  0\}}\left(  \sum _{n=-\infty}^{+\infty}
\Phi(N,n)e^{2\pi inm\cdot\alpha}  \right)\widehat{f}(m)e^{2\pi im\cdot x}.
\end{align*}

Hence, by Cauchy's inequality and assumption $(i)$, one has the estimate
\begin{align*}
\left| \mathcal D^{\Phi, \alpha}_N  f(x) \right| &\leq \left( \sum _{m\in\mathbb{Z}^{d}\setminus\{  0\}}\vert m\vert^{-2\delta} \left\vert \sum _{n=-\infty}^{+\infty}\Phi(N,n)  e^{2\pi inm\cdot\alpha}\right\vert ^{2}\right)  ^{\frac{1}{2}}\|f\|_{\delta, 2} \\
&\leq K\left(\sum _{m\in\mathbb{Z}^{d}\setminus\{  0\}}\vert m\vert ^{-2\delta}(1+N\Vert m\cdot\alpha\Vert)^{-2\vartheta}\right)^{\frac{1}{2}}\|f\|_{\delta, 2}.
\end{align*}

By Lemma \ref{L-2}, for every positive integer $M$ one has
\begin{align*}
\sum _{m\in\mathbb{Z}^{d}\setminus\{  0\}}\vert m\vert ^{-2\delta}&
(1+N\Vert m\cdot\alpha\Vert)^{-2\vartheta}\\
&\leq cN^{-2\vartheta}\sum _{k=0}^{M-1}2^{-2\delta k}\left(\sum _{2^{k}\leq\vert m\vert <2^{k+1}}\Vert m\cdot\alpha\Vert^{-2\vartheta}\right)  +\sum_{\vert m\vert \geq2^{M}}\vert m\vert^{-2\delta}\\
&\leq\begin{cases}
cN^{-2\vartheta}\sum\limits_{k=0}^{M-1}2^{(2\vartheta\sigma+d-2d\vartheta-2\delta)k}+2^{(d-2\delta)M}  & \text{if }0<\vartheta<1/2,\\
cN^{-1}\sum\limits_{k=0}^{M-1}(1+k)2^{(\sigma-2\delta)k}+2^{(d-2\delta)M}  & \text{if }\vartheta=1/2,\\
cN^{-2\vartheta}\sum\limits_{k=0}^{M-1}2^{(2\vartheta\sigma-2\delta)k}+2^{(d-2\delta)M}  & \text{if }1/2<\vartheta<+\infty.
\end{cases}
\end{align*}

For $\vartheta<1/2$, the choice $M=\log_{2}(N)/(\sigma-d)$ if $\delta<\vartheta\sigma+d/2-d\vartheta$ or $M=2\vartheta\log_{2}(N)/(2\delta-d)$ if $\delta\geq\vartheta\sigma+d/2-d\vartheta$ gives
\begin{align*}
N^{-2\vartheta}\sum _{k=0}^{M-1}2^{(2\vartheta\sigma+d-2d\vartheta-2\delta)k}+2^{(d-2\delta)M}
\leq\begin{cases}
cN^{-\frac{2\delta-d}{\sigma-d}} & \text{if }\delta<\vartheta\sigma+d/2-d\vartheta,\\
cN^{-2\vartheta}\log(N)  & \text{if }\delta=\vartheta\sigma+d/2-d\vartheta,\\
cN^{-2\vartheta} & \text{if }\delta>\vartheta\sigma+d/2-d\vartheta.
\end{cases}
\end{align*}

Observe that if $\sigma=d$, then $\delta>\vartheta\sigma+d/2-d\vartheta$. If $\vartheta=1/2$, the choice $M=\log_{2}(N)/(\sigma-d)$ if $\delta\leq\sigma/2$ or $M=\log_{2}(N)/(2\delta-d)$ if $\delta\geq\sigma/2$ gives 
\begin{align*}
N^{-1}\sum _{k=0}^{M-1}(1+k)2^{(\sigma-2\delta)k}+2^{(d-2\delta)M}
\leq\begin{cases}
cN^{-\frac{2\delta-d}{\sigma-d}}\log(N)& \text{if }\delta<\sigma/2,\\
cN^{-1}\log^{2}(N)  & \text{if }\delta=\sigma/2,\\
cN^{-1} & \text{if }\delta>\sigma/2.
\end{cases}
\end{align*}

Observe that if $\sigma=d$, then $\delta>\sigma/2$. For the last case $\vartheta>1/2$, the choice $M=2\vartheta\log_{2}(N)/(2\vartheta\sigma-d)$ if $\delta\leq\sigma\vartheta$ or $M=2\vartheta\log_{2}(N)/(2\delta-d)$ if $\delta\geq\vartheta\sigma$ gives
\begin{align*}
N^{-2\vartheta}\sum _{k=0}^{M-1}2^{(2\vartheta\sigma-2\delta)k}+2^{(d-2\delta)M}
&\leq\begin{cases}
cN^{-\frac{2\vartheta(2\delta-d)}{2\vartheta\sigma-d}}& \text{if }\delta<\vartheta\sigma,\\
cN^{-2\vartheta}\log(N)  & \text{if }\delta=\vartheta\sigma,\\
cN^{-2\vartheta} & \text{if }\delta>\vartheta\sigma.
\end{cases}
\end{align*}

Collecting the above estimates one obtains that
\[
\left( \sum _{m\in\mathbb{Z}^{d}\setminus\{  0\}}\vert m\vert^{-2\delta}(N\Vert m\cdot\alpha\Vert)^{-2\vartheta}\right)^{\frac{1}{2}} \leq cN^{-\vartheta}X(d,\delta,\vartheta,\sigma,N).
\]
\end{proof}

\begin{proof}[Proof Theorem \ref{T-3}]
Recall that the Fourier coefficients are bounded by the $L^{1}(\mathbb{T}^d)$ norm of the function, so that 
\begin{align*}
\sup_{x\in \mathbb{T}^d}
\left\{ \left|\mathcal D^{\Phi, \alpha}_N f(x) \right|\right\} &\geq\int _{\mathbb{T}^{d}} \left| \mathcal D^{\Phi,\alpha}_Nf(x) \right| dx \\
&\geq \left\vert \sum _{n=-\infty}^{+\infty}\Phi(N,n)  e^{2\pi inm\cdot\alpha}\right\vert \vert \widehat{f}(m) 
\vert  .
\end{align*}

Therefore,
\begin{align*}
&\limsup_{N\to+\infty}
\left\{ N^{\vartheta}\sup_{x\in \mathbb{T}^d}
\left\{ \left| \mathcal D^{\Phi, \alpha}_N f(x) \right| \right\} \right\} \geq \limsup_{N\to+\infty}\left\{N^{\vartheta}\int _{\mathbb{T}^{d}}\left| \mathcal D^{\Phi,\alpha}_N f(x)\right| dx\right\}\\
&\geq \limsup_{N\to +\infty}\left\{N^\vartheta \left\vert \sum _{n=-\infty}^{+\infty}\Phi(N,n)  e^{2\pi inm\cdot\alpha}\right\vert  \right\} | \widehat{f}(m)|
=C(m\cdot \alpha) \vert \widehat{f}(m)\vert.
\end{align*}
\end{proof}
The proof of Theorem \ref{T-4} is straightforward. We include a proof for the sake of completeness.
\begin{proof}[Proof of Theorem \ref{T-4}]
It suffices to recall that the Sobolev space $W^{\frac{d}{2},2}(\mathbb{T}^d)$ contains unbounded functions. If $f(x)$ is unbounded in just one point and if $\alpha$ is an irrational vector, or if the weights $\Phi(N,n)$ are non-negative and $\alpha$ is arbitrary, then in the sum $\sum_{n=-\infty}^{+\infty}\Phi(N,n)  f(x+n\alpha)$ the possible infinite terms do not cancel. Hence \[\sup_{x\in\mathbb{T}^{d}}\left\{\left|\mathcal D^{\Phi, \alpha}_N f(x) \right|\right\} =+ \infty.\]

An explicit example of function in $W^{\frac{d}{2},2}(\mathbb{T}^d)$ unbounded in a neighborhood of the origin and bounded elsewhere is given by the series 
\[
f(x) = \sum_{m\in\mathbb Z^d} \left(1+|m|^2 \right)^{-\frac{d}{2}}\log^{-1} \left(2+|m| \right) e^{2\pi im\cdot x}.
\]
\end{proof}

The following lemma is a main ingredient in the proof of Theorem \ref{T-5}.

\begin{lemma}\label{L-3}
\begin{itemize} \item[$(i)$] If $\vartheta >0$, then, for every $\alpha\in\mathbb R^d$,
one has
\[
\sum_{m\in\mathbb Z^d\backslash\{0\}}|m|^{-d\vartheta}
\|m\cdot\alpha \|^{-\vartheta} =+\infty.
\]
\item[$(ii)$] If $\vartheta >0$, then, for almost every $\alpha\in\mathbb R^d$, one has
\[
\sum_{m\in\mathbb Z^d\backslash\{0\}}|m|^{-d\vartheta} \log^{-\vartheta}\left(1+|m|\right)
\|m\cdot\alpha \|^{-\vartheta} =+\infty.
\]
\end{itemize}
\end{lemma}

\begin{proof} $(i)$ It is a classical result of Dirichlet in Diophantine approximation that for every vector $\alpha\in\mathbb{R}^d$ and every positive integer $M$ there exists $m=\left( m_{1},m_{2},\dots,m_{d} \right)$ in $\mathbb{Z}^{d}$ with $\left| m_{j}\right| \leq M$ for every $j=1,\ldots,d$, and with $\| \alpha\cdot m \| \leq M^{-d}$. 
See e.g. \cite[Chapter II, Theorem 1E]{Schmidt}.
Since $|m|\leq \sqrt{d}M$, it follows that in the series we are interested in there are infinitely many terms larger than $d^{\frac{d\vartheta}{2}}$, and the series diverges.

$(ii)$ It is a classical result of Khintchine in dimension one, and of Groshev in dimension $d\geq 1$, that for almost every vector $\alpha\in\mathbb{R}^d$ there exists infinitely many $m \in\mathbb{Z}^d \setminus \{ 0 \}$ such that 
$ \|m\cdot\alpha \| \leq |m|^{-d} \log^{-1} \left( 1+|m| \right) $.
See e.g. \cite[ Chapter III, Theorem 3A]{Schmidt} and \cite{BeresnevichVelani, Sprind}.
Hence, for almost every $\alpha$ there are infinitely many terms larger than $1$ in the given series, so that such series diverges.
\end{proof}

\begin{proof}[Proof of Theorem \ref{T-5}]
Observe that if $\vartheta \leq 1$ then $d\vartheta-d/2 \leq d/2$, and this case is already covered by Theorem \ref{T-4}. 
In order to prove $(i)$, define
\[
f(x)=\sum_{m\in\mathbb Z^d\backslash\{0\}}|m|^{-d\vartheta} e^{2\pi im\cdot x}.
\]

The norm of this function in the Sobolev space 
$W^{\delta,2}(\mathbb{T}^{d})$ is
\[
\|f\|_{\delta, 2}= \left(\sum_{m\in\mathbb{Z}^{d}\setminus\{0\}}
\left(1+\vert m\vert ^{2}\right)^{\delta}
\left| m \right|^{-2 d\vartheta} \right)^{\frac{1}{2}}
\]

Hence, this function is in $W^{\delta,2}(\mathbb{T}^{d})$ if and only if $2\delta-2d\vartheta < -d$, that is, if and only if $\delta < d\vartheta - d/2$. 

\medskip

Let us estimate $\mathcal D^{\Phi, \alpha}_Nf(x)$ with $x=0$,
  
\begin{align*}
\mathcal D^{\Phi, \alpha}_Nf(0)
&=\sum _{m\in\mathbb{Z}^{d}\setminus\{  0\}} \left(  
\sum _{n=-\infty}^{+\infty} \Phi(N,n) e^{2\pi i m\cdot \alpha} \right)|m|^{-d\vartheta}
\\
& \geq H\sum _{m\in\mathbb{Z}^{d}\setminus\{  0\}}
|m|^{-d\vartheta} \left(1+N \|m\cdot\alpha \| \right)^{-\vartheta} \\&\geq H 2^{-\vartheta}N^{-\vartheta}\sum _{\|m\cdot\alpha \| > \frac{1}{N}}
|m|^{-d\vartheta} \|m\cdot\alpha \|^{-\vartheta} .
\end{align*}
Then, by part $(i)$ of Lemma \ref{L-3} , it follows that, for every $\alpha$, one has
\begin{align*}
&\limsup_{N\to+\infty} \left\{N^{\vartheta}
\sup_{x\in\mathbb{T}^{d}}\left\{\left|\mathcal D^{\Phi, \alpha}_N f(x) \right|\right\} \right\}  \geq H \limsup_{N\to+\infty} \left\{ 
\sum _{\|m\cdot\alpha \| > \frac{1}{N}}
|m|^{-d\vartheta} \|m\cdot\alpha \|^{-\vartheta} \right\}
=+ \infty.
\end{align*}

The proof of $(ii)$ is similar. Define 
\[
f(x)=\sum_{m\in\mathbb Z^d\backslash\{0\}}|m|^{-d\vartheta} \log^{-\vartheta} \left( 1+|m| \right) e^{2\pi im\cdot x}.
\]
If $\vartheta >1/2$ this function is in the Sobolev space 
$W^{d\vartheta -\frac{d}{2},2}(\mathbb{T}^{d})$, and 
\begin{align*}
\mathcal D^{\Phi, \alpha}_Nf(0)
&=\sum _{m\in\mathbb{Z}^{d}\setminus\{  0\}} \left(  
\sum _{n=-\infty}^{+\infty} \Phi(N,n) e^{2\pi i m\cdot \alpha} \right)|m|^{-d\vartheta} \log^{-1} \left( 1+|m| \right)
\\
& \geq H\sum _{m\in\mathbb{Z}^{d}\setminus\{  0\}}
 \left(1+N \|m\cdot\alpha \| \right)^{-\vartheta} |m|^{-d\vartheta} \log^{-1} \left( 1+|m| \right)\\
& \geq H 2^{-\vartheta} N^{-\vartheta}\sum _{\|m\cdot\alpha \| > 1/N}
  |m|^{-d\vartheta} \log^{-1} \left( 1+|m| \right) \|m\cdot\alpha \|^{-\vartheta} .
\end{align*}
Then, by part $(ii)$ of Lemma \ref{L-3}, it follows that, for almost every $\alpha$, one has 
\begin{align*}
&\limsup_{N\to+\infty} \left\{N^{\vartheta}
\sup_{x\in\mathbb{T}^{d}}\left\{\left|\mathcal D^{\Phi, \alpha}_N f(x) \right|\right\} \right\} \\
& \geq H 2^{-\vartheta} \limsup_{N\to+\infty} \left\{ 
\sum _{\|m\cdot\alpha \| > \frac{1}{N}}
|m|^{-d\vartheta} \log^{-1} ( 1+|m|)
\|m\cdot\alpha \|^{-\vartheta} \right\}
=+ \infty.
\end{align*}
\end{proof}
At last, we prove Theorem \ref{T-6}.
\begin{proof}[Proof of Theorem \ref{T-6}]
Let $A$ be a subset of $\Omega$ of cardinality $|A|<+\infty$, and let 
\[ f(x)= \sum_{m \in A}\left(1+|m|^2\right)^{-\frac{\delta}{2}}e^{2\pi im\cdot x}. \]

Then $\|f\|_{\delta,2}=|A|^{\frac{1}{2}}$. Moreover, for the $N$'s in the theorem and under the assumption that $\Vert \alpha\cdot m\Vert \leq L\vert m\vert^{-\sigma}$ for every $m \in A$, 
\begin{align*}
\mathcal{D}_N^{\Phi,\alpha} f(0) 
&= \sum_{m \in A} \left(1+|m|^2\right)^{-\frac{\delta}{2}}\left( \sum _{n=-\infty}^{+\infty}\Phi(N,n)e^{2\pi inm\cdot\alpha}\right) \\
& \geq H \sum_{m \in A} \left(1+|m|^2\right)^{-\frac{\delta}{2}} 
\left(1+N\|m\cdot\alpha\|\right)^{-\vartheta}\\
& \geq H \sum_{m \in A} \left(1+|m|^2\right)^{-\frac{\delta}{2}} 
\left(1+NL|m|^{-\sigma}\right)^{-\vartheta}\\
& \geq 2^{-\vartheta-\frac{\delta}{2}}H L^{-\vartheta }N^{-\vartheta}\sum_{m \in A,\, |m|\leq (NL)^{\frac{1}{\sigma}}} |m|^{-\delta - \vartheta \sigma}.
\end{align*}

Hence, if $\delta \leq \vartheta \sigma$ one has 
\begin{align*}
\limsup_{N\to+\infty} \left\{N^{\vartheta}
\sup_{x\in\mathbb{T}^{d}}\left\{\left|\mathcal D^{\Phi, \alpha}_N f(x) \right|\right\} \right\} 
& \geq 2^{-\vartheta-\frac{\delta}{2}}H L^{-\vartheta }\sum_{m \in A} |m|^{-\delta - \vartheta \sigma} \\
& \geq  2^{-\vartheta-\frac{\delta}{2}}H L^{-\vartheta } |A| 
= 2^{-\vartheta-\frac{\delta}{2}}H L^{-\vartheta } |A|^{\frac{1}{2}}\|f\|_{\delta,2} .
\end{align*}

Letting $|A| \to+\infty $, it follows that the family of operators $ \{ D^{\Phi, \alpha}_N \}_{N=1}^{+\infty} $ is not uniformly bounded from $W^{\vartheta \sigma, 2}(\mathbb T^2)$ into $L^\infty(\mathbb T^d)$.
Therefore, by the resonance theorem of Banach and Steinhaus, there exists a function $f\in W^{\vartheta \sigma, 2}(\mathbb T^d)$ such that 
\[
\limsup_{N\to+\infty}\left\{N^{\vartheta}\sup_{x\in\mathbb T^d}\left\{\left| \mathcal D^{\Phi,\alpha}f(x) \right|\right\}\right\}=+\infty.
\]
\end{proof}

We conclude the section proving the corollaries.
\begin{proof}[Proof of Corollary \ref{C-1}]
Recall that, as observed in the introduction, if the theorems apply with an exponent $\vartheta_0$, then they also apply with every $\vartheta_1 \leq \vartheta_0$.  
The choice of $\Psi(t)=\chi_{\{-1\leq t\leq 1\}}(t)$ gives
\begin{align*}
\frac {1}{2N+1}\sum_{n=-N}^N e^{2\pi int} =\frac{\sin\big((2N+1)\pi t\big)}{(2N+1)\sin(\pi t)}.
\end{align*}

Up to a factor $1/(2N+1)$ one  recognizes the Dirichlet kernel and easily verifies that
\[
\left\vert\frac{\sin\big((2N+1)\pi t\big)}{(2N+1)\sin(\pi t)}\right\vert\leq \frac{K}{1+N\|t\|}.
\]

Hence Theorem \ref{T-1} and Theorem \ref{T-2} apply with $\vartheta=1$.
In order to prove that the speed of convergence $cN^{-1}$ cannot be accelerated one can apply Theorem \ref{T-3}. However, there is also a more elementary and general argument that applies to every nonconstant function $f(x)$. Assume that there exists a pair $\{N,N+1\}$ such that
\begin{align*}
\int_{\mathbb{T}^{d}}\left| \frac{1}{N}\sum_{n=1}^{N}f(x+n\alpha)  -\int_{\mathbb{T}^{d}}f(y)dy\right| dx&<\frac{1}{2N}
\int_{\mathbb{T}^{d}}\left| f(x)  -\int_{\mathbb{T}^{d}}f(y)dy\right| dx,\\
\int_{\mathbb{T}^{d}}\left| \frac{1}{N+1}\sum_{n=1}^{N+1}f(x+n\alpha)  -\int_{\mathbb{T}^{d}}f(y)dy\right| dx&<\frac{1}{2(
N+1)  }\int_{\mathbb{T}^{d}}\left| f(x)  -\int_{\mathbb{T}^{d}}f(y)dy\right| dx.
\end{align*}

Then the triangle inequality gives a contradiction,
\begin{align*}
&\int_{\mathbb{T}^{d}}\left| f(x)  -\int_{\mathbb{T}^{d}}f(y)dy\right| dx=\int_{\mathbb{T}^{d}}\left| f(x+(N+1)\alpha)  -\int_{\mathbb{T}^{d}}f(y)dy\right| dx \geq \\
&\int_{\mathbb{T}^{d}}\left(  
\left| \sum_{n=1}^{N}f(x+n\alpha)  -N\int_{\mathbb{T}^{d}}f(y)dy \right| +
\left| \sum_{n=1}^{N+1}f(x+n\alpha)  -(N+1)  \int_{\mathbb{T}^{d}}f(y)dy \right| \right)  dx\\
&>\int_{\mathbb{T}^{d}}\left| f(x)  -\int_{\mathbb{T}^{d}}f(y)dy\right| dx.
\end{align*}
\end{proof}

\begin{proof}[Proof of Corollary \ref{C-2}] 
In this case we have
\begin{align*}
\frac {1}{N}\sum_{n=1-N}^{N-1}\left(1-\frac{|n|}{N}\right) e^{2\pi int} =\frac{1}{N}\left(\frac{\sin(\pi N t)}{\sqrt{N}\sin(\pi t)}\right)^2.
\end{align*}

Up to a factor $1/N$ one  recognizes the Fej\'er kernel, and checks  that Theorem \ref{T-1} and Theorem \ref{T-2} apply with $\vartheta=2$.
To prove that the speed of convergence $cN^{-2}$ cannot be accelerated observe that
\[
C(t)=\limsup_{N\to +\infty} \left\{ N^{2}\left| \frac{1}{N}\left(\frac{\sin(\pi N t)}{\sqrt{N}\sin(\pi t)}\right)^2
\right| \right\}= \frac{\limsup\limits_{N\to+\infty} \left\{\sin^2(\pi N t)\right\}}{\sin^2(\pi t)}.
\]

If $t=0$, then $C(0)=+\infty$. If $t\neq0$ is rational, then $\sin^2(\pi Nt)$ takes a finite number of values for $N\to +\infty$, hence $\limsup\limits_{N\to+\infty}\{\sin^2(\pi N t)\}>0$. If $t$ is irrational then,
by Kronecker's theorem, $\limsup\limits_{N\to+\infty}\left\{\sin^2(\pi N t)\right\}=1$ . The conclusion follows from Theorem \ref{T-3}.
\end{proof}

\begin{proof}[Proof of Corollary \ref{C-3}]The function $\Psi(t)$ is related to the Bochner--Riesz kernel.
Recall the  integral representation of Bessel functions,
\[\int _{-N}^{+N}\left(  1-\dfrac{\vert t\vert ^{2}}{N^{2}}\right)^{\gamma}e^{  2\pi ist}  dt=\pi^{-\gamma}\Gamma(\gamma+1)N^{-\gamma+\frac{1}{2}}\vert s\vert ^{-\gamma-\frac{1}{2}}J_{\gamma+\frac{1}{2}}(2\pi N\vert s\vert ) .
\]

The Poisson summation formula gives the series expansion 
\[
\sum_{n=1-N}^{N-1}\left(  1-\dfrac{\vert n\vert^{2}}{N^{2}} \right)^{\gamma}e^{2\pi int}=\pi^{-\gamma}\Gamma(\gamma+1)N^{-\gamma+\frac{1}{2}}\sum_{k=-\infty}^{+\infty}\vert t+k\vert ^{-\gamma-\frac{1}{2}}J_{\gamma+\frac{1}{2}}(2\pi N\vert t+k\vert)  .
\]

See \cite[Chapter 4, Theorem 4.15 and Chapter 7, Theorem 2.4]{SW}.
Observe that the use of the Poisson summation formula is legitimate since both above series are absolutely and uniformly convergent (see \cite[Lemmas 4 and 5]{Stein}). Also recall that the Bessel function $J_{\gamma+\frac{1}{2}}(z)$ has the asymptotic expansions 
\[
J_{\gamma+\frac{1}{2}}(z) =
\begin{cases}
\dfrac{z^{\alpha}}{2^{\alpha}\Gamma(\alpha +1)} +O\left(z^{\alpha +1}\right) & \textrm{if }z\to 0^+,\\
\sqrt{\dfrac{2}{\pi z}}\sin\left(z-\pi \gamma/2\right)+O(z^{-\frac{3}{2}}) & \textrm{if } z\to+\infty.
\end{cases}
\]

Assume for simplicity that $0<t< 1/2$. Then the above sum has a main term of the form 
\[
N^{-\gamma+\frac{1}{2}}\|t\|^{-\gamma-\frac{1}{2}}J_{\gamma+\frac{1}{2}}(2\pi N\|t\|).
\]

The remainder is the sum over all $k$'s with $|t+k|\geq 1/2$ and it can be estimated as
\begin{align*}
\bigg| N^{-\gamma+\frac{1}{2}}\!\!\!\sum_{|t+k|\geq \frac{1}{2}} |t+k|^{-\gamma-\frac{1}{2}}J_{\gamma+\frac{1}{2}}(2\pi N|t+k|)\bigg|&\leq cN^{-\gamma+\frac{1}{2}}\!\!\!\sum_{|t+k|\geq \frac{1}{2}} |t+k|^{-\gamma-\frac{1}{2}}(N|t+k|)^{- \frac{1}{2}}\\&\leq cN^{-\gamma}.
\end{align*}

The main term can be estimated from above by
\begin{align*}
N^{-\gamma+\frac{1}{2}}\|t\|^{-\gamma-\frac{1}{2}}J_{\gamma+\frac{1}{2}}(2\pi N\|t\|)
&\leq
\begin{cases}
cN^{-\gamma+\frac{1}{2}}\|t\|^{-\gamma-\frac{1}{2}}\left(N\|t\|\right)^{\gamma+\frac{1}{2}} & \textrm{if } N\|t\|\leq1,\\
cN^{-\gamma+\frac{1}{2}}\|t\|^{-\gamma-\frac{1}{2}}\left(N\|t\|\right)^{-\frac{1}{2}} & \textrm{if } N\|t\|>1,
\end{cases}
\\
&\leq
\begin{cases}
cN& \textrm{if } N\|t\|\leq1,\\
cN^{-\gamma}\|t\|^{-\gamma-1}& \textrm{if } N\|t\|>1.
\end{cases}
\end{align*}

Observe that the estimates of the main term dominate the remainder. Also notice that \[
cN\leq\sum_{\vert n\vert \leq \frac{N}{2}}\left(  1-\dfrac{1}{4}\right)^{\gamma}\leq\sum_{n=1-N}^{N-1}\left(  1-\dfrac{\vert n\vert^{2}}{N^{2}}\right)^{\gamma}\leq\sum_{\vert n\vert \leq N}1\leq CN.
\]

In conclusion,
\[
\left| \left( \sum_{n=1-N}^{N-1}\left( 1-\dfrac{|n|^2}{N^2}\right)^{\gamma}\right)^{-1}\sum_{n=1-N}^{N-1}\left(1-\dfrac{\vert n\vert ^{2}}{N^{2}}\right)^{\gamma}e^{2\pi int}\right| \leq c(1+N\|t\|)^{-\gamma-1}.
\]

Hence, Theorem \ref{T-1} and Theorem \ref{T-2} apply with $\vartheta=\gamma+1$. To apply Theorem \ref{T-3} let us show that there exist $\varepsilon>0$ and $\eta>0$ such that for every $t$ in a set of measure $\eta$ one has $C(t)>\varepsilon$.  Observe that
\[
C(t) \approx \limsup_{N\to+\infty}\left\{N^{\frac{1}{2}}\left|\sum_{k=-\infty}^{+\infty}\vert t+k\vert ^{-\gamma-\frac{1}{2}}J_{\gamma+\frac{1}{2}}(2\pi N\vert t+k\vert)\right|\right\}.
\]

Again assume that $t+k$ is not an integer for every $k$. The asymptotic expansion of Bessel functions gives 
\begin{align*}
&N^{\frac{1}{2}}\left|\sum_{k=-\infty}^{+\infty}\vert t+k\vert ^{-\gamma-\frac{1}{2}}J_{\gamma+\frac{1}{2}}(2\pi N\vert t+k\vert)\right|\\
&\geq N^{\frac{1}{2}}\|t\|^{-\gamma-\frac{1}{2}}\left| J_{\gamma+\frac{1}{2}}(2\pi N\|t\|) \right| -N^{\frac{1}{2}}\sum_{|t+k|\geq \frac{1}{2}} |t+k|^{-\gamma-\frac{1}{2}}\left|J_{\gamma+\frac{1}{2}}(2\pi N|t+k|) \right|\\
&\geq\pi^{-1} \|t\|^{-\gamma-1}\left|\sin\left(2\pi N\|t\|-\frac{\pi\gamma }{2}\right)\right|- cN^{\frac{1}{2}}\|t\|^{-\gamma-\frac{1}{2}}(N\|t\|)^{-\frac{3}{2}}\\
&\qquad-cN^{\frac{1}{2}}\sum_{|t+k|\geq \frac{1}{2}} |t+k|^{-\gamma-\frac{1}{2}}(N|t+k|)^{-\frac{1}{2}}\\
&\geq \pi^{-1} \|t\|^{-\gamma-1}\left|\sin\left(2\pi N\|t\|-\frac{\pi\gamma }{2} \right)\right|-cN^{-1}\|t\|^{-\gamma-2}-c\\
&= \pi^{-1}\|t\|^{-\gamma-1}\left(\left|\sin\left(2\pi N\|t\|-\frac{\pi\gamma }{2}\right)\right|-cN^{-1}\|t\|^{-1}-c\|t\|^{\gamma+1}\right).
\end{align*}

Let $0<\lambda<1/2$. Then for every $t$ such that $\lambda/2\leq\|t \|\leq \lambda$
one has
\begin{align*}
&\left|\sin\left(2\pi N\|t\|- \frac{\pi\gamma }{2}\right)\right|-cN^{-1}\|t\|^{-1}-c\|t\|^{\gamma+1}\\&\geq 
\left|\sin\left(2\pi N\|t\|- \frac{\pi\gamma }{2}\right)\right| -2c\lambda^{-1}N^{-1}-c\lambda^{\gamma+1}.
\end{align*}

In conclusion, if $\lambda$ is suitably small, for every $N$ suitable large one has 
\[2c\lambda^{-1}N^{-1}+c\lambda^{\gamma+1}< \frac{1}{2}. \]

Moreover, if $t$ is irrational, by Kronecker's theorem, 
\[
\limsup_{N\to+\infty}\left\{ \left|\sin\left(2\pi N\|t\|-\pi\gamma /2\right)\right| \right\}=1.
\]

Hence $C(t) >c>0$ for every $t$.
\end{proof}

\begin{proof}[Proof of Corollary \ref{C-4}] 
One has 
\begin{align*}
\sum_{n=-N}^{N}
\dfrac{(2N)!}{2^{2N} \left( N-n \right) ! \left( N+n \right) !} e^{2\pi i nt}
&= \sum_{m=0}^{2N}
\dfrac{(2N)!}{m! \left( 2N-m \right) ! } 
\left( \dfrac{ e^{\pi i t}}{2}\right)^m
\left( \dfrac{ e^{-\pi i t}}{2}\right)^{2N-m}\\
 &= \left( \dfrac{e^{\pi i t}+e^{-\pi i t}}{2} \right)^{2N}
 =\cos^{2N}(\pi t).
\end{align*}
 
Up to a normalizing factor one recognizes the de la Vall\'ee Poussin kernel, which is similar to the heat kernel. Indeed, when $t\to 0$,  
\[
\cos^{2N}(\pi t)=e^{2N\log \left(\cos (\pi t)\right)}
=e^{2N\log \left(1-\frac{\pi^2 t^2}{2}+...\right)}
\approx e^{-\pi^2 Nt^2}.
\]

It follows that, for every $\vartheta >0$, there exists $c>0$ such that 
\[
\cos^{2N}(\pi t)\leq ce^{-\pi^2 Nt^2} 
\leq c ( 1+\sqrt{N} \|t\| )^{-2\vartheta}.
\]

Hence, Theorem \ref{T-1} and Theorem \ref{T-2} apply with $N$ replaced by $[\sqrt{N}]$.
\end{proof}

\begin{proof}[Proof of Corollary \ref{C-5}] 
The fact that the speed of convergence $cN^{-\frac{\delta}{d}}$ is the best possible in the Sobolev space $W^{\delta,2}(\mathbb{T}^{d})$ is proved in \cite{BCCGST}. The fact that one can actually obtain such speed of convergence, up to some small transgression, follows from Theorem \ref{T-2} with suitable combinations of $\delta$, $\vartheta$ and $\sigma$. 
Let $ \Phi(N,n) = \left(\sum _{n=-\infty}^{+\infty}\Psi(N^{-1}n)  \right)^{-1}\Psi(N^{-1}n)$, with  $\Psi(t)$ a  smooth compactly supported bounded function satisfying 
$cN\leq \left| \sum _{n=-\infty}^{+\infty}\Psi(N^{-1}n) \right| \leq CN$.
Then, for every $\vartheta$ there exists $c$ such that \[
\left| \sum_{n=-\infty}^{+\infty}\Phi(N,n)  e^{2\pi int} \right| \leq c(  1+N\Vert t\Vert)^{-\vartheta}
\]
for every positive integer $N$.
To see this set  
\[
\Lambda(s)  =\int _{-\infty}^{+\infty}\Psi(t)  e^{ 2\pi ist}  dt.
\]

Since $\Psi(t)$ is smooth with compact support, by iterated integration by parts, one has
\[
\left|\int _{-\infty}^{+\infty} \Psi(t)  e^{2\pi ist} dt\right| \leq (2\pi\vert s\vert)^{-j}\int _{-\infty}^{+\infty}
\left| \dfrac{d^{j}}{dt^{j}}\Psi(t)  \right| dt.
\]

Hence, for every $\vartheta$ there exists $c$ such that
\[
\vert \Lambda(s)\vert \leq c(  1+\vert s\vert)  ^{-\vartheta}.
\]

By the Poisson summation formula,
\[
\sum_{n=-\infty}^{+\infty}\Psi(N^{-1}n)  e^{2\pi int} 
=  N\sum_{n=-\infty}^{+\infty} \Lambda(N(n+t)) .
\]

It follows that, for every $t\neq0$,
\begin{align*}
\left| \sum_{n=-\infty}^{+\infty}\Psi(N^{-1}n)  e^{2\pi int}  \right| &\leq N\sum_{n=-\infty}^{+\infty}\vert \Lambda(N(n+t)) \vert \\
&\leq cN(1+N\Vert t\Vert)^{-\vartheta}+cN\sum_{n\neq 0}(1+N\vert n\vert)^{-\vartheta}\\
&\leq cN(1+N\Vert t\Vert)^{-\vartheta}+cN^{1-\vartheta}\leq cN(  1+N\Vert t\Vert)^{-\vartheta}.
\end{align*}

Hence, for this $\Phi(N,n)$ assumption $(i)$ in Theorem \ref{T-2} holds with arbitrary $\vartheta$, and for fixed $\sigma$ and $\delta$ one can chose a $\vartheta$ that optimizes the estimates in Theorem \ref{T-2}. In particular, if $\delta/\sigma\geq 1/2$  one can choose $\vartheta=\delta/\sigma $, whereas if $\delta/\sigma<1/2$ one can choose 
$(\delta-d/2)/(\sigma-d)<\vartheta<1/2 $. 
\end{proof}

\section{Concluding remarks}\label{final_section}

\begin{remark}\label{R-1}
The above corollaries show that the assumptions in Theorem \ref{T-1} and Theorem \ref{T-2} are not void. We now show that the assumptions in in Theorem \ref{T-5} and Theorem \ref{T-6} are not void as well. Let us prove that for every $\vartheta >0$ there exists a positive weight $\Phi(N,n)$ which satisfies \eqref{condizione}, with the property that $n \rightarrow \Phi(N,n)$ has compact support for every $N$, and with the property that there exist constants $H>0$ and $K>0$ such that for every $t$ one has 
\[
H(1+N\Vert t\Vert)^{-\vartheta} \leq
 \sum _{n=-\infty}^{+\infty}\Phi(N,n)e^{2\pi int}
 \leq K(1+N\Vert t\Vert)^{-\vartheta}.
\]

Let $j$ be a positive integer, and let 
\[ 
F_N (t) = (1+N\Vert t\Vert)^{-\vartheta},\;\;\;\;\;\;
G_N(t) = N^{1-2j}\left(\frac{\sin(\pi N t)}{\sin(\pi t)}\right)^{2j},
\]
\[
F_N *G_N(t) = \int_{\mathbb{T}}F_N (t-s) G_N(s) ds.
\]

It is easily verified that $G_N(t)$ is a trigonometric polynomial of degree 
$(N-1)j$. Hence, the convolution $F_N *G_N(t)$ is a trigonometric polynomial as well. From the inequalities $2\Vert t\Vert \leq | \sin(\pi t ) | \leq \pi \Vert t \Vert$ one deduces that
\[
\begin{cases}
2^{2j} \pi^{-2j} N \leq  G_N (t) \leq 2^{-2j} \pi^{2j} N & \text{if } \Vert t\Vert \leq 1/(2N),\\
0 \leq G_N (t) \leq 2^{-2j} N^{1-2j} \Vert t \Vert ^{-2j} & \text{if } \Vert t\Vert \geq 1/(2N).
\end{cases}
\]

It follows that there exist constants $C>c>0$ such that for every $N$ one has \[c\leq  \int_{\mathbb{T}} G_N(t) dt \leq C.\]

In order to estimate $F_N *G_N(t) $ from below, observe that for every $t$,
\[
\int_{\mathbb{T}}F_N (t-s) G_N (s) ds
\geq 2^{2j} \pi^{-2j} N \int_{\left\{\| s \| < \frac{1}{2N} \right\} }  
(1+N\Vert t-s \Vert)^{-\vartheta} ds
\geq c (1+N\Vert t \Vert)^{-\vartheta} .
\]

In order to estimate $F_N *G_N(t) $ from above, observe that for every $t$,
\[
\int_{\mathbb{T}} F_N (t-s) G_N (s) ds \leq 
\sup_{s\in \mathbb{T}} \left\{
F_N( t-s ) \right\}
\int_{\mathbb{T}} G_N(t) dt \leq c.
\]

Moreover, if $\Vert t\Vert \geq \/1/N$ then 
\begin{align*}
 \int_{\mathbb{T}}F_N (t-s) G_N (s) ds 
& \leq 2^{-2j} \pi^{2j} N 
\left(1+N\left( \Vert t \Vert -\dfrac{1}{2N} \right) \right)^{-\vartheta} \int_{\left\{\| s \| < \frac{1}{2N} \right\} }  ds \\
&+ 2^{-2j} N^{1-2j} 
\left(1+\dfrac{N \Vert t \Vert}{2} \right)^{-\vartheta}
\int _{ \left\{ \frac{1}{2N} \leq \| s \| < \frac{\Vert t\Vert}{2} \right\} } \Vert s \Vert ^{-2j}ds \\
& + 2^{-2j} N^{1-2j} \int _{ \left\{ \frac{\Vert t\Vert }{2} \leq \| s \| \leq \frac{1}{2} \right\} }  \Vert s \Vert ^{-2j}ds 
\\
&
\leq c (N\Vert t \Vert)^{-\vartheta} + c (N\Vert t \Vert)^{-\vartheta} +c (N\Vert t \Vert)^{1-2j } .
\end{align*}

Hence, if $2j-1 \geq \vartheta $ then for some positive constants $c$ and $C$ independent of $N$ and for every $t$ one has 
\[
c (1+N\Vert t \Vert)^{-\vartheta} \leq
\int_{\mathbb{T}}F_N (t-s) G_N (s) ds
\leq C (1+N\Vert t \Vert)^{-\vartheta} .
\]

Finally, define $\Phi(N,n)$ as the Fourier transform of 
$ \left( F_N *G_N(0) \right)^{-1} F_N *G_N(t)$,
\[
\Phi(N,n) = \left( F_N *G_N(0) \right)^{-1} 
\int_{\mathbb{T}} F_N *G_N(t) e^{-2\pi i nt}dt.
\]
\end{remark}

\bigskip

\begin{remark}\label{delta_sharp_2}
Observe that we cannot apply Theorem \ref{T-5} and Theorem \ref{T-6} to Corollary \ref{C-1} and Corollary \ref{C-2}, since in these corollaries $\sum _{n=-\infty}^{+\infty}\Phi(N,n)e^{2\pi int}$ vanishes in many points.
However, the ranges of indexes $\delta$ of the Sobolev class $W^{\delta, 2}(\mathbb T^d)$ in Corollary \ref{C-1}, and in Corollary \ref{C-2} at least in dimension one, seem to be essentially sharp as well. 
In Corollary \ref{C-1} with $\vartheta=1$ the assumption $\delta>d\vartheta-d/2$ in Theorem \ref{T-1} becomes $\delta>d/2$. As already observed, this assumption is necessary since the Sobolev spaces $W^{\delta,2}(\mathbb{T}^d)$ with $\delta\leq d/2$ contain unbounded functions. In Corollary \ref{C-2} with $\vartheta=2$ the assumption $\delta>d\vartheta-d/2$ becomes $\delta>3d/2$. At least, in dimension $d=1$ one can prove that this range is essentially sharp, in the sense that it cannot be replaced by any index $\delta<3/2$. Indeed, for every $\delta<3/2$ there exists a function $f(x)\in W^{\delta,2}(\mathbb{T})$ such that 
\[\limsup_{N\to+\infty} \left\{N^2\sup_{x\in\mathbb{T}}\left\{\left|\mathcal{D}^{\Phi,\alpha}_{N}f(x)\right|\right\} \right\}=+\infty\]
for almost every $\alpha$.
In order to show that this is true, observe that 
\begin{equation}\label{below_estimate}
N^2\sup_{x\in\mathbb{T}}\left\{\left|\mathcal{D}^{\Phi,\alpha}_{N}f(x)\right|\right\}\geq \left|\mathcal{D}^{\Phi,\alpha}_{N}f(0)\right| =\left|\sum_{m\neq 0} \dfrac{\sin^{2}(\pi Nm\alpha)}{\sin^{2}(\pi m\alpha)}\widehat{f}(m)\right|.
\end{equation}

In Petersen \cite{Petersen} it is proved that if $0<\alpha,\beta<1$ the following are equivalent:
\begin{enumerate}
\item[(i)] $\beta\in \mathbb{Z}\alpha$ (mod 1);
    \item[(ii)] $\sup\limits_{N\geq 1}\left\{ \displaystyle\sum_{m=1}^{+\infty} \dfrac{\|m\beta\|^2}{m^2}\dfrac{\|Nm\alpha\|}{\|m\alpha\|^2} \right\}<+\infty$.
\end{enumerate}

It is easy to verify that (ii) is also equivalent to 
\begin{enumerate}
    \item[(iii)] $\sup\limits_{N\geq 1}\left\{ \displaystyle\sum_{m=1}^{+\infty} \dfrac{\|m\beta\|^2}{m^2}\dfrac{\sin^{2}(\pi Nm\alpha)}{\sin^{2}(\pi m\alpha)}\right\}<+\infty$.
\end{enumerate}

Let $\beta\in (0,1)$ be an algebraic number and set 
\[f(x)=\sum_{m\neq 0} \dfrac{\|m\beta\|^2}{m^2}e^{2\pi imx}. \]

Such a function $f(x)$ belongs to $W^{\delta,2}(\mathbb{T})$ for every $\delta<3/2$.
Since for every transcendental number $\alpha\in (0,1)$ condition (i) does not holds, then (iii) does not hold as well. This is exactly what we wanted to prove thanks to \eqref{below_estimate} and the fact that almost every $\alpha\in(0,1)$ is a transcendental number.  

\end{remark}

\begin{remark} 
It is curious to compare the above results on the speed of convergence in
ergodic theorems with the approximation properties of Fourier series. Whereas our results suggest that stronger summation methods guarantee faster convergence, the approximation properties of partial sums and F\'ejer means of Fourier series seem to go in the opposite direction. Assume $d=1$ and denote by $\mathcal{S}_{N}f(x)  $ and $\mathcal{F}_{N}f(x)  $ the partial sums and the arithmetic means of the partial sums of the Fourier expansion
of a function $f(x)$,
\[
\mathcal{S}_{N}f(x)  =\sum _{m=-N}^{+N}\widehat{f}(m)  e^{2\pi im\cdot x},\qquad
\mathcal{F}_{N}f(x)  =\sum _{m=-N}^{+N}\Big(  1-\dfrac{\vert m\vert }{N+1}\Big)  \widehat{f}(m)  e^{2\pi im\cdot x}  .
\]

The partial sums $\mathcal{S}_{N}f(x)  $ may not converge, but the approximation is close to optimal. Indeed, if $\|S_N\|$ denotes the operator norm of the partial sums, the Lebesgue constant, if $\mathcal{E}_{N}(f)$ denotes the best approximation in the supremum norm of $f(x)$ with trigonometric polynomials of degree at most $N$, and if $P_{N}(x)$ is the trigonometric polynomial of best approximation, then  
\begin{align*}
\sup_{x\in\mathbb{T}}\left\{\vert \mathcal{S}_{N}f(x)  -f(x) \vert \right\} &=\sup_{x\in\mathbb{T}}\left\{\vert \mathcal{S}_{N}(f-P_N)(x) -(f(x)-P_N(x)) \vert \right\}\\&  \leq (\|S_N\|+1)\sup_{x\in\mathbb{T}}\{|f(x)-P_N(x)|\}\leq  c\mathcal{E}_{N}(f)  \log(1+N)  .
\end{align*}

Finally, the means $\mathcal{F}_{N}f(x)  $ always converge, but the approximation is never better than $c/N$,
\[
\sup_{x\in\mathbb{T}}\left\{\vert \mathcal{F}_{N}f(x)  -f(x)\vert \right\}  \geq\int_{\mathbb{T}}\vert \mathcal{F}_{N}f(x)  -f(x)  \vert dx\geq\dfrac{\vert m\vert }{N+1}\vert\widehat{f}(m)\vert .
\]

In particular, the partial sums may converge faster than the F\'ejer means.
\end{remark}

\section{Appendix}\label{appendix}
In this appendix we deal with logarithmic means. Such means are defined by the weights 
\[ \Phi(N,n) =\begin{cases} \left(\displaystyle \sum\limits_{m=1}^{N}\dfrac{1}{m} \right)^{-1} \dfrac{1}{n} & \text{if } 1\leq n\leq N,\\  0 & \text{otherwise},\end{cases}\] 
  and the associated logarithmic discrepancy is 
 \[
\mathcal D^{\Phi,\alpha}_N f(x)=
\left(\sum _{n=1}^{N} \dfrac{1}{n} \right)^{-1}\sum_{n=1}^{N}\dfrac{f(x+n\alpha)}{n}  -\int_{\mathbb{T}^{d}}f(y)dy.
\]
 
See \cite[Section 2.2]{DT} for references about these means and discrepancy.
Although Theorem \ref{T-1} and Theorem \ref{T-2} do not immediately apply in this setting,
due to the fact that the assumption $(i)$ on the kernels $\sum _{n=-\infty}^{+\infty}\Phi(N,n)e^{2\pi int}$ is not satisfied, the proofs can be adapted to obtain some analogues of the above results.

\begin{Theorem}\label{T-7}
 If the function $f(x)$ has an absolutely convergent Fourier expansion, 
 $\sum_{m\in\mathbb Z}|\widehat f(m)|<+\infty$, then, for almost every $\alpha$, there exists a positive constant $c(f,\alpha)$ such that, for every positive integer $N$, one has
\[
\sup_{x\in\mathbb{T}^{d}}\left\{\left|\mathcal D^{\Phi, \alpha}_N f(x) \right|\right\}\leq 
c(f,\alpha) \log^{-1}(1+N).
\]
\end{Theorem}

\begin{Theorem}\label{T-8}
If $\alpha$ is not a Liouville vector, that is, if there exist positive constants $H$ and $\sigma $ such that $\Vert \alpha\cdot m\Vert \geq H\vert m\vert^{-\sigma}$ for every $m\in\mathbb{Z}^{d}\setminus\{0\}$, then there exists a positive constant $c=c(d,H,\sigma)$, such that for every positive integer $N$, 
\[
\sup_{x\in\mathbb{T}^{d}}\left\{\left|\mathcal D^{\Phi, \alpha}_N f(x) \right|\right\}\leq 
c \log^{-1}(1+N) 
\sum_{m\in\mathbb Z}|\widehat f(m)| \log\left(1+|m| \right).
\]
\end{Theorem}

In particular, the above theorems apply to functions in Sobolev classes $W^{\delta,2}(\mathbb{T}^d)$ with $\delta>d/2$.
The main ingredient in the proofs of both theorems is an estimate for the kernels $\sum _{n=-\infty}^{+\infty}\Phi(N,n)e^{2\pi int}$.

\begin{lemma}\label{L-4} For every $t \in \mathbb{T}$,
\[
\left| \left(\sum _{n=1}^{N} \dfrac{1}{n} \right)^{-1}\sum_{n=1}^{N} \dfrac{e^{2\pi i n t}}{n}  \right|
\leq \min \left\{1, c \dfrac{\log \left( 1 + \Vert t \Vert ^{-1} \right)}{\log \left( 1 + N \right)} \right\}.
\]

Moreover, if $N$ is large enough and if $\Vert t \Vert \geq c/N$, the reverse inequality holds true as well.
\end{lemma}

\begin{proof}

A direct and explicit proof goes as follows. 
The inequality $\leq 1$ is obvious. In order to prove the other inequality it suffices to assume that $N>1$ and $ |t| \leq 1/2$. An integration by parts gives
\begin{align*}
&\sum_{n=1}^{N} \dfrac{e^{2\pi i n t}}{n} 
= \dfrac{1}{N} \sum_{n=1}^{N} e^{2\pi i n t} +
\sum_{n=1}^{N-1} \left( \dfrac{1}{n} - \dfrac{1}{n+1} \right) \sum_{m=1}^{n}e^{2\pi i m t}
\\
& 
= \dfrac{\sin\left(\pi N t \right)}{N \sin\left(\pi t \right)} e^{i \pi (N+1) t} + 
\sum_{n=1}^{N-1}  \dfrac{1}{n+1}  
\dfrac{\sin\left(\pi n t \right)}{n \sin\left(\pi t \right)} e^{i \pi (n+1) t}\\
&
=  \dfrac{\sin\left(\pi N t \right)}{N \sin\left(\pi t \right)} e^{i \pi (N+1) t} 
+ \sum_{n=1}^{N-1}  \dfrac{1}{n+1}  
\dfrac{\sin\left(\pi n t \right)}{n \sin\left(\pi t \right)} \left( e^{i \pi (n+1) t} -1 \right)
+ \sum_{n=1}^{N-1}  \dfrac{1}{n+1}  
\dfrac{\sin\left(\pi n t \right)}{n \sin\left(\pi t \right)}\\
&=I+I\!I+I\!I\!I.
\end{align*}

The inequalities $2 \Vert t \Vert \leq \left\vert \sin\left( \pi t\right)  \right\vert \leq \pi \Vert t \Vert $ imply that
\[
\left\vert \frac{\sin\left( \pi n t \right) }{ n \sin \left(  \pi t \right) } \right\vert \leq
\min \left\{\dfrac{\pi}{2},\dfrac{1}{2n \Vert t \Vert } \right\},
\]
from which an estimate for the term $I$ is immediately obtained.
In order to estimate $I\!I$ one can separately consider the sum where the index $n$ varies in the set $\left\{ 1\leq n \leq N-1,\; n \leq 1/|t| \right\}$
and the sum where the index varies in the set 
$\left\{ 1\leq n \leq N-1,\; n > 1/|t| \right\}$.
The latter set is empty if $|t|<1/(N-1)$. Otherwise, there is a uniform bound. Indeed, 
\[
\left| \sum_{ 1\leq n \leq N-1,\; n > 1/|t| }  \dfrac{1}{n+1}  
\dfrac{\sin\left(\pi n t \right)}{n \sin\left(\pi t \right)} \left( e^{i \pi (n+1) t} -1 \right) \right| 
\leq \dfrac{1}{|t|}
\sum_{n>1/|t| }  \dfrac{1}{n(n+1)} \leq c.
\]

The inequality $\left| e^{i \pi (n+1) t} -1 \right| \leq \pi (n+1) |t|$ implies that also the sum over the $\left\{ 1\leq n \leq N-1,\; n \leq 1/|t| \right\}$ is uniformly bounded. Indeed,
\[
\left| \sum_{ 1\leq n \leq N-1,\; n \leq 1/|t| }  \dfrac{1}{n+1}  
\dfrac{\sin\left(\pi n t \right)}{n \sin\left(\pi t \right)} \left( e^{i \pi (n+1) t} -1 \right) \right| 
\leq \dfrac{\pi^2}{2}|t|
\sum_{n \leq 1/|t| }  1 \leq c.
\]

In order to estimate the last term $I\!I\!I$, one considers separately the sum over the set of indexes $\left\{ 1\leq n \leq N-1,\; n \leq 1/(2|t|) \right\}$
and the sum over the set of indexes
$\left\{ 1\leq n \leq N-1,\; n > 1/(2|t|) \right\}$.
The sum over this latter set is uniformly bounded,
\[
\left| \sum_{ 1\leq n \leq N-1,\; n > 1/(2|t|) }  \dfrac{1}{n+1}  
\dfrac{\sin\left(\pi n t \right)}{n \sin\left(\pi t \right)}  \right| 
\leq \dfrac{1}{2|t|}
\sum_{n>1/(2|t|) }  \dfrac{1}{n(n+1)} \leq c.
\]

The sum over the indexes $\left\{ 1\leq n \leq N-1,\; n \leq 1/(2|t|) \right\}$ is bounded by 
\[
\left| \sum_{ 1\leq n \leq N-1,\; n \leq 1/(2|t|) }  \dfrac{1}{n+1}  
\dfrac{\sin\left(\pi n t \right)}{n \sin\left(\pi t \right)}  \right| \leq 
\dfrac{\pi}{2} \sum_{ n \leq 1/(2|t|) }  \dfrac{1}{n+1}  
\leq C \log\left(1+|t|^{-1} \right).
\]

Notice that if $| t | \geq 1/(2N-2)$, then the reverse inequality holds true,
\[
\left| \sum_{ 1\leq n \leq N-1,\; n \leq 1/(2|t|) }  \dfrac{1}{n+1}  
\dfrac{\sin\left(\pi n t \right)}{n \sin\left(\pi t \right)}  \right| \geq 
\dfrac{2}{\pi} \sum_{ n \leq 1/(2|t|) }  \dfrac{1}{n+1}  
\geq c \log\left(1+|t|^{-1} \right).
\]
\end{proof}

\begin{proof}[Proof of Theorem \ref{T-7}] 
As in the proof of Theorem \ref{T-1}, 
\begin{align*}
\mathcal D^{\Phi, \alpha}_Nf(x)& =
\left(\sum _{n=1}^{N} \dfrac{1}{n} \right)^{-1}\sum_{n=1}^{N}\dfrac{f(x+n\alpha)}{n}  -\int_{\mathbb{T}^{d}}f(y)dy\\
&
= \left( \sum _{n=1}^{N} \dfrac{1}{n} \right)^{-1} \sum _{m\in\mathbb{Z}^{d}\setminus\{ 0\}} \left( \sum _{n=1}^{N} \dfrac{e^{2\pi i n m\cdot\alpha}}{n} \right) \widehat{f}(m)e^{2\pi im\cdot x}.
\end{align*}

Hence, by Lemma \ref{L-4},
\[
\left| \mathcal D^{\Phi, \alpha}_Nf(x) \right|
\leq c \log^{-1} (1+N) \sum _{m\in\mathbb{Z}^{d}\setminus\{ 0\}} | \widehat{f}(m) | \log \left( 1 + \Vert m \cdot \alpha \Vert ^{-1} \right).
\]

And, by Lemma \ref{L-1},
\begin{align*}
& \int_{\mathbb{T}^{d}} \left( 
\sum _{m\in\mathbb{Z}^{d}\setminus\{ 0\}} | \widehat{f}(m) | \log \left( 1 + \Vert m \cdot \alpha \Vert ^{-1} \right) \right) d\alpha\\&= \left( \int_{\mathbb{T}} \log( 1 + \Vert t \Vert ^{-1}) dt \right) \sum _{m\in\mathbb{Z}^{d}\setminus\{ 0\}} | \widehat{f}(m) |.
\end{align*}

Finally, $ \displaystyle \int_{\mathbb{T}} \log \left( 1 + \Vert t \Vert ^{-1} \right) dt \ < +\infty$.
\end{proof}
 
\begin{proof}[Proof of Theorem \ref{T-8}] 
As in the proof of Theorem \ref{T-7}, under the Diophantine assumption $\Vert \alpha\cdot m\Vert \geq H\vert m\vert^{-\sigma}$,
\begin{align*}
\left| \mathcal D^{\Phi, \alpha}_Nf(x) \right|
& \leq c \log^{-1} (1+N) \sum _{m\in\mathbb{Z}^{d}\setminus\{ 0\}} | \widehat{f}(m) | \log \left( 1 + \Vert m \cdot \alpha \Vert ^{-1} \right)\\
& \leq c \log^{-1} (1+N) \sum _{m\in\mathbb{Z}^{d}\setminus\{ 0\}} | \widehat{f}(m) | \log \left( 1 + H^{-1} |m|^{\sigma} \right)\\
& \leq c \log^{-1} (1+N) \sum _{m\in\mathbb{Z}^{d}\setminus\{ 0\}} | \widehat{f}(m) | \log \left( 1 + |m| \right).
\end{align*}
\end{proof}

\bibliography{WeylSums-bib}

\begin{thebibliography}{10}

\bibitem{BBH}
{\sc Bayart, F., Buczolich, Z., and Heurteaux, Y.}
\newblock Fast and slow points of {B}irkhoff sums.
\newblock {\em Ergodic Theory Dynam. Systems 40}, 12 (2020), 3236--3256.

\bibitem{BeresnevichVelani}
{\sc Beresnevich, V., and Velani, S.}
\newblock Classical metric {D}iophantine approximation revisited: the
  {K}hintchine-{G}roshev theorem.
\newblock {\em Int. Math. Res. Not. IMRN}, 1 (2010), 69--86.

\bibitem{BCCGST}
{\sc Brandolini, L., Choirat, C., Colzani, L., Gigante, G., Seri, R., and
  Travaglini, G.}
\newblock Quadrature rules and distribution of points on manifolds.
\newblock {\em Ann. Sc. Norm. Super. Pisa Cl. Sci. (5) 13}, 4 (2014), 889--923.

\bibitem{Colzani}
{\sc Colzani, L.}
\newblock Speed of convergence in an ergodic theorem.
\newblock {\em submitted\/} (2022).

\bibitem{Colzani2}
{\sc Colzani, L.}
\newblock Speed of convergence of {W}eyl sums over {K}ronecker sequences.
\newblock {\em Monatsh. Math. https://doi.org/10.1007/s00605-022-01785-z\/}
  (2022).

\bibitem{DP}
{\sc Dick, J., and Pillichshammer, F.}
\newblock Periodic functions with bounded remainder.
\newblock {\em Arch. Math. (Basel) 87}, 6 (2006), 554--563.

\bibitem{DT}
{\sc Drmota, M., and Tichy, R.~F.}
\newblock {\em Sequences, discrepancies and applications}, vol.~1651 of {\em
  Lecture Notes in Mathematics}.
\newblock Springer-Verlag, Berlin, 1997.

\bibitem{EEGGP}
{\sc Ehler, M., Etayo, U., Gariboldi, B., Gigante, G., and Peter, T.}
\newblock Asymptotically optimal cubature formulas on manifolds for prefixed
  weights.
\newblock {\em J. Approx. Theory 271\/} (2021), 24 pages.

\bibitem{GG}
{\sc Gariboldi, B., and Gigante, G.}
\newblock Optimal asymptotic bounds for designs on manifolds.
\newblock {\em Anal. PDE 14}, 6 (2021), 1701--1724.

\bibitem{GL2}
{\sc Grepstad, S., and Larcher, G.}
\newblock Sets of bounded remainder for a continuous irrational rotation on
  {$[0,1]^2$}.
\newblock {\em Acta Arith. 176}, 4 (2016), 365--395.

\bibitem{GL1}
{\sc Grepstad, S., and Lev, N.}
\newblock Sets of bounded discrepancy for multi-dimensional irrational
  rotation.
\newblock {\em Geom. Funct. Anal. 25}, 1 (2015), 87--133.

\bibitem{HL}
{\sc Hellekalek, P., and Larcher, G.}
\newblock On {W}eyl sums and skew products over irrational rotations.
\newblock {\em Theoret. Comput. Sci. 65}, 2 (1989), 189--196.

\bibitem{KP}
{\sc Kakutani, S., and Petersen, K.}
\newblock The speed of convergence in the {E}rgodic {T}heorem.
\newblock {\em Monatsh. Math. 91}, 1 (1981), 11--18.

\bibitem{Kesten}
{\sc Kesten, H.}
\newblock On a conjecture of {E}rd{\H{o}}s and {S}z{\"{u}}sz related to uniform
  distribution mod 1.
\newblock {\em Acta Arith. 12\/} (1966), 193--212.

\bibitem{Krengel}
{\sc Krengel, U.}
\newblock On the speed of convergence in the {E}rgodic {T}heorem.
\newblock {\em Monatsh. Math. 86}, 1 (1978), 3--6.

\bibitem{KN}
{\sc Kuipers, L., and Niederreiter, H.}
\newblock {\em Uniform distribution of sequences}.
\newblock Pure and Applied Mathematics. Wiley-Interscience [John Wiley \&
  Sons], New York-London-Sydney, 1974.

\bibitem{Lang}
{\sc Lang, S.}
\newblock {\em Introduction to {D}iophantine approximations}, second~ed.
\newblock Springer-Verlag, New York, 1995.

\bibitem{Petersen}
{\sc Petersen, K.}
\newblock On a series of cosecants related to a problem in ergodic theory.
\newblock {\em Compositio Math. 26\/} (1973), 313--317.

\bibitem{Schmidt}
{\sc Schmidt, W.~M.}
\newblock {\em Diophantine approximation}, vol.~785 of {\em Lecture Notes in
  Mathematics}.
\newblock Springer, Berlin, 1980.

\bibitem{Sprind}
{\sc Sprind\v{z}uk, V.~G.}
\newblock {\em Metric theory of {D}iophantine approximations}.
\newblock Scripta Series in Mathematics. V. H. Winston \& Sons, Washington,
  D.C.; John Wiley \& Sons, New York-Toronto, Ont.-London, 1979.
\newblock Translated from the Russian and edited by Richard A. Silverman, With
  a foreword by Donald J. Newman.

\bibitem{Stein}
{\sc Stein, E.~M.}
\newblock On certain exponential sums arising in multiple {F}ourier series.
\newblock {\em Ann. of Math. (2) 73\/} (1961), 87--109.

\bibitem{SW}
{\sc Stein, E.~M., and Weiss, G.}
\newblock {\em Introduction to {F}ourier analysis on {E}uclidean spaces}.
\newblock Princeton Mathematical Series, No. 32. Princeton University Press,
  Princeton, N.J., 1971.

\bibitem{Travaglini}
{\sc Travaglini, G.}
\newblock {\em Number theory, {F}ourier analysis and geometric discrepancy},
  vol.~81 of {\em London Mathematical Society Student Texts}.
\newblock Cambridge University Press, Cambridge, 2014.

\bibitem{Z}
{\sc Zygmund, A.}
\newblock {\em Trigonometric series. {V}ol. {I}, {II}}, third~ed.
\newblock Cambridge Mathematical Library. Cambridge University Press,
  Cambridge, 2002.
\newblock With a foreword by Robert A. Fefferman.

\end{thebibliography}
 \bibliographystyle{acm}

\end{document}